\documentclass[a4paper]{amsart}
\usepackage[lmargin=3.932cm,rmargin=3.932cm,tmargin=4.1cm,bmargin=3.4cm]{geometry}
\usepackage[english]{babel}
\usepackage{amsmath,amsthm,amssymb,amsfonts}
\usepackage{bbm}
\usepackage{enumitem}
\setlist[enumerate]{leftmargin=*}

\usepackage{hyperref}

\newtheorem{thm}{Theorem}[section]
\newtheorem{lem}[thm]{Lemma}
\newtheorem{prp}[thm]{Proposition}
\newtheorem{cor}[thm]{Corollary}

\theoremstyle{definition}
\newtheorem{rem}[thm]{Remark}

\newcommand{\statement}[3]{\begin{gather}\tag{#1}\label{#2}\parbox{\dimexpr\linewidth-4em}{#3}\end{gather}}

\newcommand{\CC}{\mathbb{C}}
\newcommand{\RR}{\mathbb{R}}
\newcommand{\NN}{\mathbb{N}}
\newcommand{\ZZ}{\mathbb{Z}}

\newcommand{\TT}{\mathbb{T}}

\newcommand{\Rpos}{{\mathbb{R}^+}}
\newcommand{\Rnon}{{\mathbb{R}^+_0}}

\newcommand{\meas}{\mu}
\newcommand{\Pmeas}{\sigma}
\newcommand{\GPmeas}{\varpi}

\newcommand{\tc}{:}

\DeclareMathOperator{\Span}{span}         

\newcommand{\opL}{\mathfrak{L}}           

\newcommand{\opML}{\mathfrak{M}}          

\newcommand{\Sz}{\mathcal{S}}             

\newcommand{\szP}{\mathcal{N}}                    

\newcommand{\Diff}{\mathfrak{D}}

\newcommand{\Hilb}{\mathcal{H}}
\newcommand{\HS}{\mathrm{HS}}

\newcommand{\irr}{\mathrm{irr}}

\newcommand{\Eigen}{\mathcal{P}}
\newcommand{\Val}{\vartheta}

\newcommand{\Four}{\mathcal{F}}

\newcommand{\Kern}{\mathcal{K}}

\newcommand{\group}[1]{\mathrm{#1}}       

\newcommand{\lie}{\mathfrak}

\newcommand{\Mot}{{\group{SE}(2)}}
\newcommand{\mot}{{\lie{se}(2)}}

\newcommand{\chr}{\mathbf{1}}

\numberwithin{equation}{section}

\begin{document}
\title[Sub-Laplacians on groups of polynomial growth]{Convolution kernels versus spectral multipliers for sub-Laplacians on groups of polynomial growth}

\author[A. Martini]{Alessio Martini}
\address[A. Martini]{School of Mathematics\\ University of Birmingham\\Edgbaston\\Birmingham B15 2TT \\ United Kingdom}
\email{a.martini@bham.ac.uk}
\author[F. Ricci]{Fulvio Ricci}
\address[F. Ricci]{Scuola Normale Superiore\\ Piazza dei Cavalieri 7\\56126 Pisa\\ Italy}
\email{fricci@sns.it}
\author[L. Tolomeo]{Leonardo Tolomeo}
\address[L. Tolomeo]{School of Mathematics \\The University of Edinburgh\\James Clerk Maxwell Building\\ King's Buildings \\ Edinburgh EH3 9JZ\\United Kingdom}
\email{L.Tolomeo@sms.ed.ac.uk}

\thanks{The first-named author is a member of the Gruppo Nazionale per l'Analisi Matematica, la Probabilit\`a e le loro Applicazioni (GNAMPA) of the Istituto Nazionale di Alta Matematica (INdAM). This work was partially supported by the EPSRC Grant ``Sub-Elliptic Harmonic Analysis'' (EP/P002447/1).}

\begin{abstract}
Let $\opL$ be a sub-Laplacian on a connected Lie group $G$ of polynomial growth. It is well known that, if $F : \RR \to \CC$ is in the Schwartz class $\Sz(\RR)$, then the convolution kernel $\Kern_{F(\opL)}$ of the operator $F(\opL)$ is in the Schwartz class $\Sz(G)$. Here we prove a sort of converse implication for a class of groups $G$ including all solvable noncompact groups of polynomial growth. We also discuss the problem whether integrability of $\Kern_{F(\opL)}$ implies continuity of $F$.
\end{abstract}

\maketitle

\section{Introduction}

Among the most fundamental and useful properties of the Fourier transformation $\Four$ on $\RR^n$ one can certainly include the \emph{invariance of the Schwartz class} (a function $f$ is in the Schwartz class of smooth and rapidly decaying functions on $\RR^n$ if and only if its Fourier transform $\Four f$ is in the Schwartz class) and the \emph{Riemann--Lebesgue Lemma} (if $f$ is an integrable function on $\RR^n$, then $\Four f$ is continuous and vanishing at infinity). It is natural to ask whether these properties hold in more general contexts than $\RR^n$, where an analogue of the Fourier transformation can be defined.

If $G$ is a connected Lie group of polynomial growth, then the Schwartz class on $G$ can be defined in a natural way. On the other hand, the group Fourier transformation, based on unitary representation theory, is a much more complicated construct, which is not immediately amenable to notions of smoothness (the group Fourier transform of an integrable function on $G$ is an operator-valued function on the ``dual object'' $\widehat{G}$, which in general has no differentiable manifold structure). Studies of the Schwartz class ``on the Fourier transform side'' are available in the literature for particular groups \cite{astengo_fourier_2013,geller_fourier_1977,geller_fourier_1980,ludwig_fourier_2007}, but developing a general theory appears to be a very difficult problem.

The matter simplifies in case one restricts to subclasses of functions on $G$ which are commutative under convolution, where scalar-valued analogues of the Fourier transform, such as the Gelfand transform, are available. A considerable attention has been given to the case of the spherical Fourier transform on Gelfand pairs, where the subclass of functions is determined by a compact group of symmetries such that the corresponding convolution algebra is commutative \cite{astengo_gelfand_2007,astengo_gelfand_2009,benson_spherical_1998,fischer_gelfand_2009,fischer_nilpotent_2012,fischer_nilpotent_preprint2012,fischer_nilpotent_2013,fischer_nilpotent_preprint2017}. However the existence of such a group of symmetries is a fairly strong constraint, considerably reducing the examples that one can consider. On the other hand, it is possible to state in great generality and without symmetry constraints (other than translation-invariance) the problem of the invariance of the Schwartz class in relation to the functional calculus for a sub-Laplacian (or, more generally, for a subelliptic system of commuting differential operators \cite{martini_multipliers_2010,martini_spectral_2011}).

Let $G$ be a connected Lie group of polynomial growth with fixed Haar measure $\meas$. Let $\{X_j\}_j$ be a system of left-invariant vector fields on $G$ that satisfy H\"ormander's condition. Let $\opL = -\sum_j X_j^2$ be the corresponding sub-Laplacian.

The operator $\opL$ is essentially self-adjoint and nonnegative on $L^2(G)$ and has a functional calculus defined via the spectral theorem. For all bounded Borel functions $F : \Rnon \to \CC$, where $\Rnon = [0,\infty)$, let $\Kern_{F(\opL)}$ denote the convolution kernel of the operator $F(\opL)$. In general, $\Kern_{F(\opL)}$ is only a distribution on $G$. However, as soon as $F$ is compactly supported, $K_{F(\opL)} \in L^2(G)$, and there exists a regular Borel measure $\Pmeas_\opL$ on $\Rnon$, whose support is the $L^2$-spectrum $\Sigma_\opL$ of $\opL$, such that the ``Plancherel identity''
\begin{equation}\label{eq:opLplancherel}
\int_G |\Kern_{F(\opL)}(x)|^2 \,d\meas(x) = \int_\Rnon |F(\lambda)|^2 \,d\Pmeas_\opL(\lambda)
\end{equation}
holds (see, e.g., \cite[Theorem 3.10]{martini_spectral_2011}); we call $\Pmeas_\opL$ the Plancherel measure associated with the sub-Laplacian $\opL$.
It should be noted that $\Kern_{F(\opL)}$ depends only on the restriction $F|_{\Sigma_\opL}$, and $\Kern_{F(\opL)} = 0$ if and only if $F$ vanishes $\sigma_\opL$-almost everywhere.
 In particular the map $T_\opL : F \mapsto \Kern_{F(\opL)}$ extends to an isometric isomorphism from $L^2(\Rnon, \Pmeas_\opL)$ to a closed subspace $\Gamma_\opL^2$ of $L^2(G)$, and moreover $T_\opL$ maps $L^1(\Rnon,\Pmeas_\opL)$ into the space $C_0(G)$ of continuous functions on $G$ vanishing at infinity.

The analogies of the ``kernel transform'' $T_\opL$ with the Fourier transform do not end here.
Let $\Sz(\Rnon)$ denote the space of restrictions to $\Rnon = [0,\infty)$ of elements of $\Sz(\RR)$; the space $\Sz(\Rnon)$ has a natural Fr\'echet structure as a quotient of $\Sz(\RR)$.
Then the following result holds.

\begin{thm}[Hulanicki \cite{hulanicki_functional_1984}]\label{thm:hulanicki}
Let $G$ be a connected Lie group of polynomial growth and $\opL$ be a sub-Laplacian thereon. If $F : \Rnon \to \CC$ is in the Schwartz class $\Sz(\Rnon)$, then $\Kern_{F(\opL)}$ is in the Schwartz class $\Sz(G)$.
\end{thm}

Actually \cite{hulanicki_functional_1984} discusses in detail the particular case of homogeneous Lie groups and homogeneous operators $\opL$, but it is possible to adapt the argument to any sub-Laplacian on a Lie group of polynomial growth (see also \cite{alexopoulos_spectral_1994}, \cite[Proposition 4.2.1]{martini_multipliers_2010} and \cite[Theorem 6.1(iii)]{martini_crsphere}).

In this paper we discuss the validity of the following statement, which can be thought of as a converse to Theorem \ref{thm:hulanicki}:
\statement{A}{st:szweak}{If $F : \Rnon \to \CC$ is continuous and $\Kern_{F(\opL)} \in \Sz(G)$, then $F \in \Sz(\Rnon)$.}

We are able to prove this statement for a certain class of groups of polynomial growth.
Let $\Mot$ denote the plane motion group, that is, the semidirect product $\RR^2 \rtimes \TT$ where the torus $\TT$ acts on $\RR^2$ by rotations.

\begin{thm}\label{thm:mainmain}
Let $G$ be a connected Lie group of polynomial growth and $\opL$ be a sub-Laplacian thereon. Statement \eqref{st:szweak} holds whenever the group $G$ has a quotient isomorphic to either $\RR$ or $\Mot$.
\end{thm}

\begin{cor}\label{cor:main}
Let $G$ be a connected Lie group of polynomial growth and $\opL$ be a sub-Laplacian thereon. Statement \eqref{st:szweak} holds whenever the group $G$ is solvable and noncompact.
\end{cor}

Note that the validity of statement \eqref{st:szweak} implies, in particular, that the $L^2$-spectrum of $\opL$ is the whole $\Rnon$
 (otherwise one could modify $F$ outside $\Sigma_\opL$ so that $F$ is continuous but not smooth on $\Rnon$, without changing $\Kern_{F(\opL)}$). 
Under the assumption that $\Sigma_\opL = \Rnon$, statement \eqref{st:szweak} can be equivalently rephrased as follows:
\statement{A'}{st:szweak2}{If $F : \Rnon \to \CC$ is continuous and $\Kern_{F(\opL)} \in \Sz(G)$, then $F$ coincides $\Pmeas_\opL$-almost everywhere with an element of $\Sz(\Rnon)$.}
Statement \eqref{st:szweak2} itself does not force $\Sigma_\opL$ to be the whole $\Rnon$ and indeed may hold on groups (such as $\TT^n$) where the spectrum of $\opL$ is discrete.

One may wonder whether the \emph{a priori} continuity assumption on $F$ in statements \eqref{st:szweak} and \eqref{st:szweak2} is really needed. Strictly speaking, the continuity assumption in \eqref{st:szweak} is necessary, since one can modify $F$ on a $\Pmeas_\opL$-null set without changing $K_{F(\opL)}$.
On the other hand, one can wonder whether this is just an issue of choosing the right representative of $F$ modulo $\Pmeas_\opL$, i.e., whether the following statement holds:
\statement{B}{st:szstrong}{If $F : \Rnon \to \CC$ is a bounded Borel function and $\Kern_{F(\opL)} \in \Sz(G)$, then $F$ coincides $\Pmeas_\opL$-almost everywhere with an element of $\Sz(\Rnon)$.}

This stronger statement would immediately follow from statement \eqref{st:szweak2} and the following statement, which we can think of as an analogue of the Riemann--Lebesgue lemma:
\statement{C}{st:rl}{If $F : \Rnon \to \CC$ is a bounded Borel function and $\Kern_{F(\opL)} \in L^1(G)$, then  then $F$ coincides $\Pmeas_\opL$-almost everywhere with an element of $C_0(\Rnon)$.}

Note that the fact that $F$ vanishes at infinity, i.e.,
\begin{equation}\label{eq:vanishinfty}
\lim_{R \to \infty} \| F \, \chr_{[R,\infty)} \|_{L^\infty(\Rnon,\Pmeas_\opL)} = 0,
\end{equation}
whenever $\Kern_{F(\opL)} \in L^1(G)$, can be proved in great generality and follows from the fact that the heat kernel associated to $\opL$ is an approximate identity (see, e.g., \cite[Proposition 3.14]{martini_spectral_2011}); what is difficult to prove is the continuity of $F$ (or rather, the existence of a continuous representative modulo $\Pmeas_\opL$).

The above ``Riemann--Lebesgue lemma'' \eqref{st:rl} for the functional calculus of a sub-Laplacian can be checked in many particular cases,
however we are not aware of any result of this kind for arbitrary Lie groups $G$ of polynomial growth and sub-Laplacians $\opL$. The techniques developed in this paper allow us to show the validity of the ``Riemann--Lebesgue lemma'' (and consequently of the strengthened version of Theorem \ref{thm:mainmain}) in a number of cases.

\begin{thm}\label{thm:riemannlebesgue}
Let $G$ be a connected Lie group of polynomial growth and $\opL$ be a sub-Laplacian thereon. Statement
\eqref{st:rl} holds in each of the following cases:
\begin{enumerate}[label=\textup{(\roman*)}]
\item\label{en:riemannlebesguehomogeneous} $G$ is a stratified Lie group, and $\opL$ is a homogeneous sub-Laplacian;
\item\label{en:riemannlebesgueabelian} $G$ is abelian;
\item\label{en:riemannlebesguemot} $G$ is the plane motion group $\Mot$.
\end{enumerate}
\end{thm}

Part \ref{en:riemannlebesguehomogeneous} 
of Theorem \ref{thm:riemannlebesgue} actually follows from a stronger result, namely, the adjoint $T_\opL^*$ of the kernel transform $T_\opL$ maps $L^1(G)$ into $C(\Sigma_\opL)$; note that $T_\opL^*$ extends $T_\opL^{-1} : \Gamma_\opL^2 \to L^2(\Rnon,\Pmeas_\sigma)$, since $T_\opL$ is an isometry.
However, as we shall see, this stronger statement may fail for other groups and sub-Laplacians (even on abelian groups), without preventing \eqref{st:rl} from being true.

The proof of Theorem \ref{thm:riemannlebesgue}\ref{en:riemannlebesguemot}, instead, makes fundamental use of detailed knowledge of the representation theory and the group Plancherel measure for $\Mot$, and it does not seem easy to extend the method of proof to essentially more general classes of groups.
It would seem natural to find a more ``conceptual'' approach to the result, however there appear to be several obstacles; we will discuss some of them in Sections \ref{s:kerneltransform} and \ref{s:groupplancherel}.

The reduction to quotients in the proof of Theorem \ref{thm:mainmain}, as well as the subsequent analysis based on representation theory, exploits the following fundamental result proved in \cite[Proposition 2.1]{ludwig_sub-laplacians_2000} (see also \cite[Proposition 1.1]{mller_restriction_1990} and \cite[Proposition 3.7]{martini_spectral_2011}), whose validity depends on the amenability of $G$.

\begin{prp}\label{prp:rep_fc_cont}
Let $F : \Rnon \to \CC$ be continuous and such that $\Kern_{F(\opL)} \in L^1(G)$. Then, for all unitary representations $\pi$ of $G$,
\begin{equation}\label{eq:rep_fc_cont}
\pi(\Kern_{F(\opL)}) = F(d\pi(\opL)).
\end{equation}
\end{prp}

Here, for all unitary representations $\pi$ of $G$ and $K \in L^1(G)$, $\pi(K)$ denotes the operator $\int_G K(x) \, \pi(x^{-1}) \,d\meas(x)$, while $d\pi(\opL)$ is the image of $\opL$ via the differentiated representation $d\pi$; it is known that $d\pi(\opL)$ is essentially self-adjoint on the space of smooth vectors of the representation (see, e.g., \cite{nelson_representation_1959} or \cite[\S 3.1]{martini_spectral_2011} for details), so a functional calculus for $d\pi(\opL)$ is defined via the spectral theorem, and  Proposition \ref{prp:rep_fc_cont} gives a connection between the functional calculus of the sub-Laplacian $\opL$ and that of the corresponding operator $d\pi(\opL)$ in any unitary representation $\pi$.

The continuity assumption on $F$ in Proposition \ref{prp:rep_fc_cont} in general cannot be removed and corresponds to the continuity assumption in our Theorem \ref{thm:mainmain}. In Section \ref{s:groupplancherel} below we shall discuss a weakening of the continuity assumption, yielding a weaker version of Proposition \ref{prp:rep_fc_cont}, in the case $G$ is of type I, by exploiting the Plancherel formula for the group Fourier transform. Unfortunately this does not seem enough to obtain a Riemann--Lebesgue lemma in great generality, but will anyway be of use in the proof of Theorem \ref{thm:riemannlebesgue}\ref{en:riemannlebesguemot} for the group $\Mot$.

\subsection*{Structure of the paper}

In Section \ref{s:quotients} we recall basic definitions and results about Lie groups of polynomial growth, including the definition of the Schwartz class, and we show how the proof of statement \eqref{st:szweak} for a group $G$ reduces to the corresponding statement for some quotient of $G$, and how Corollary \ref{cor:main} follows from Theorem \ref{thm:mainmain}. In Section \ref{s:se2} we complete the proof of Theorem \ref{thm:mainmain} by proving statement \eqref{st:szweak} for the groups $\RR$ and $\Mot$. In Section \ref{s:kerneltransform} we investigate properties of the integral kernel of the kernel transform $T_\opL$ and prove parts \ref{en:riemannlebesguehomogeneous} and \ref{en:riemannlebesgueabelian} of Theorem \ref{thm:riemannlebesgue}. In Section \ref{s:groupplancherel} we discuss the relation between the group Plancherel measure and the Plancherel measure associated with a sub-Laplacian and obtain a version of Proposition \ref{prp:rep_fc_cont} with weaker assumptions and conclusion. The proof of Theorem \ref{thm:riemannlebesgue}\ref{en:riemannlebesguemot} is discussed in Section \ref{s:rlmot}. Finally, in Section \ref{s:mathieu} we discuss properties of eigenvalues and eigenfunctions of Mathieu operators that are used throughout the paper in the analysis of sub-Laplacians in irreducible representations of $\Mot$.

\subsection*{Notation}

We set $\Rnon = [0,\infty)$ and $\Rpos = (0,\infty)$. For a locally compact topological space $X$, we denote by $C(X)$, $C_c(X)$, $C_0(X)$ the spaces of functions on $X$ which are, respectively, continuous, continuous with compact support, continuous and vanishing at infinity. Moreover we write $\chr_S$ for the characteristic function of $S$.

\section{Lie groups of polynomial growth and their quotients}\label{s:quotients}

Let $G$ be a locally compact group with left Haar measure $\mu$. Recall that $G$ is said to be of polynomial growth if there exists a compact symmetric neighbourhood $U$ of the identity of $G$ which generates $G$ and such that the sequence $(\mu(U^n))_{n \in \NN}$
has at most polynomial growth as $n \to \infty$. As it turns out, this growth property does not depend on the choice of $U$, and moreover any group of polynomial growth is unimodular and amenable \cite{guivarch_croissance_1973}.
The following statement summarises a number of well-known facts (see, e.g., \cite{guivarch_croissance_1973,hewitt_abstract_1979,varadarajan_lie_1974}) that will be of use later.

\begin{prp}\label{prp:quotient_measure}
Let $G$ be a connected Lie group of polynomial growth with Haar measure $\mu$. Let $H$ be a closed normal subgroup and $\tilde G = G/H$ be the corresponding quotient.
\begin{enumerate}[label=\textup{(\roman*)}]
\item $\tilde G$ is a connected Lie group of polynomial growth.
\item $H$ is unimodular, and the action of $G$ on $H$ by conjugation preserves any Haar measure on $H$.
\item For any Haar measure $\tilde\meas$ on $\tilde G$, there exists a Haar measure $\meas_H$ on $H$ such that, if  $P : L^1(G) \to L^1(\tilde G)$ is the ``averaging operator'' satisfying
\begin{equation}\label{eq:averaging_op}
Pf(xH) = \int_H f(xh) \,d\meas_H(h) = \int_H f(hx) \,d\meas_H(h)
\end{equation}
for all $f \in L^1(G)$ and $x \in G$, then
\begin{equation}\label{eq:quotient_measure}
\int_{\tilde G} Pf(z) \,d\tilde\meas(z) = \int_G f(x) \,d\meas(x)
\end{equation}
for all $f \in L^1(G)$.
\end{enumerate}
\end{prp}

Let $G$ be a connected Lie group of polynomial growth, with Haar measure $\mu$. Its Lie algebra $\lie{g}$ can be identified with the space of left-invariant vector fields on $G$; correspondingly the universal enveloping algebra of $\lie{g}$ can be identified with the algebra $\Diff(G)$ of left-invariant differential operators on $G$.

We now recall the definition of the Schwartz class on $G$  (see also \cite{hulanicki_functional_1984,ludwig_algebre_1995,schweitzer_dense_1993} and \cite[\S 1.2.6]{martini_multipliers_2010}). Let $U$ be a compact symmetric neighbourhood of the identity of $G$ and define
\begin{equation}\label{eq:modulus}
\tau_U(x) = \min \{ k \in \NN \tc x \in U^n \}
\end{equation}
for all $x \in \NN$. Let $p \in [1,\infty]$. A function $f \in C^\infty(G)$ belongs to the Schwartz class $\Sz(G)$ if and only if, for all $N \in \NN$ and all $W \in \Diff(G)$, the quantity
\begin{equation}\label{eq:szseminorm}
\szP_{U,p;W,N}(f) = \| (1+\tau_U)^N W f \|_{L^p(G)}
\end{equation}
is finite. One can show that the definition of the class $\Sz(G)$ is independent of the choice of the exponent $p$ and the neighbourhood $U$, as also is the Fr\'echet structure induced by the family of seminorms $\{ \szP_{U,p; W,N} \tc N \in \NN, \, W \in \Diff(G) \}$ on $\Sz(G)$. Moreover, in the definition \eqref{eq:szseminorm} of the seminorms, the function $\tau_U$ could be equivalently replaced by the distance from the identity with respect to any ``connected distance'' on $G$ in the sense of \cite[\S III.4]{varopoulos_analysis_1992}.

We call sub-Laplacian on $G$ any left-invariant differential operator of the form $\opL = -\sum_j X_j^2$, where $\{X_j\}_j$ is a system of left-invariant vector fields satisfying H\"ormander's condition (i.e., the $X_j$ and their iterated Lie brackets span $\lie{g}$).

Let $\tilde G$ be a quotient of $G$ as in Proposition \ref{prp:quotient_measure} and let $p : G \to \tilde G$ be the canonical projection. Then the differential of $p$ at the identity defines an epimorphism of Lie algebras $p ' : \lie{g} \to \lie{\tilde g}$, which extends in turn to an algebra epimorphism $p' : \Diff(G) \to \Diff(\tilde G)$. In particular, if $\opL =- \sum_j X_j^2$ is a sub-Laplacian on $G$, then $p'(\opL) = -\sum_j p'(X_j)^2$ is a sub-Laplacian on $\tilde G$, which will be called the push-forward of $\opL$ on $\tilde G$.

As mentioned in the introduction, any sub-Laplacian is a nonnegative essentially self-adjoint operator on $L^2$, and the operators in its functional calculus are left-invariant, i.e., convolution operators.
The following result allows us to relate the calculus of a sub-Laplacian with that of its push-forward on a quotient.

\begin{lem}\label{lem:quoziente}
Let $G$ be a connected Lie group of polynomial growth, with Haar measure $\meas$. Let $\tilde G = G/H$ be a quotient of $G$, with Haar measure $\tilde\mu$, and let $P : L^1(G) \to L^1(\tilde G)$ be the averaging operator as in Proposition \ref{prp:quotient_measure}. Let $\opL$ be a sub-Laplacian on $G$ and $\tilde\opL$ its push-forward on $\tilde G$.
\begin{enumerate}[label=\textup{(\roman*)}]
\item\label{en:quoziente_sz} If $f \in \Sz(G)$, then $Pf \in \Sz(\tilde G)$, and the map $P : \Sz(G) \to \Sz(\tilde G)$ is continuous.
\item\label{en:quoziente_ker} If $F : \Rnon \to \CC$ is continuous
 and $\Kern_{F(\opL)} \in L^1(G)$, then $\Kern_{F(\tilde\opL)} = P \Kern_{F(\opL)}$.
\end{enumerate}
\end{lem}
\begin{proof}
\ref{en:quoziente_sz}. Let $U$ be a compact symmetric neighbourhood of the identity in $G$. Then its image $\tilde U$ in $\tilde G$ via the canonical projection $p : G \to \tilde G$ is a compact symmetric neighbourhood of the identity in $\tilde G$. If we define $\tau_U : G \to \Rnon$ and $\tau_{\tilde U} : \tilde G \to \Rnon$ as in \eqref{eq:modulus},
then clearly
\[
\tilde\tau(xH) \leq \tau(x)
\]
for all $x \in G$. Note moreover that, since $P$ can be written as an integral of left translations as in \eqref{eq:averaging_op}, and left-invariant operators commute with left translations, it is easily seen that, for all $f \in \Sz(G)$, $Pf$ is a smooth function on $\tilde G$ and
\[
P W f = p'(W) P f
\]
for all $f \in \Sz(G)$ and all $W\in \Diff(G)$.
In particular, for all $h \in \NN$,
\[\begin{split}
(1+\tau_{\tilde U}(xH))^N |p'(W) P f(xH)| &\leq (1+\tau_{\tilde U}(xH))^N P |Wf|(xH) \\
& \leq P |(1+\tau_U)^N Wf|(xH),
\end{split}\]
whence
\[\begin{split}
\szP_{\tilde U,1; p'(W),N}(Pf) &= \int_{\tilde G} (1+\tau_{\tilde U}(xH))^N |p'(W) P f(xH)| \,d\tilde\meas(xH) \\
&\leq \int_G (1+\tau_U(x))^N |Wf(x)| \,d\meas(x) = \szP_{U,1; W,N}(f),
\end{split}\]
where the Schwartz seminorms $\szP_{\tilde U,1; p'(W),N}$ and $\szP_{U,1; W,N}$ are defined as in \eqref{eq:szseminorm}.
Note moreover that $p' : \Diff(G) \to \Diff(\tilde G)$ is surjective. 
This shows that,
if $f \in \Sz(G)$, then $Pf \in \Sz(\tilde G)$ and the map $P : \Sz(G) \to \Sz(\tilde G)$ is continuous.

\ref{en:quoziente_ker}. 
Since $H$ is a normal subgroup of $G$, $xH = Hx$ for all $x \in G$ and the left and right quotients $G/H$ and $H \backslash G$ coincide. In particular there is a natural action of $G$ on $\tilde G$ by right translations, and a corresponding unitary representation $\pi$ of $G$ on $L^2(\tilde G)$. It is then easily verified that smooth vectors in the representation $\pi$ are smooth functions $f$ on $\tilde G$, and
\[
d\pi(W) f = p'(W) f
\]
for all $W \in \Diff(G)$. In particular, if $F : \Rnon \to \CC$ is continuous and $\Kern_{F(\opL)} \in L^1(G)$, then
\[
\pi(\Kern_{F(\opL)}) = F(d\pi(\opL)) = F(\tilde\opL)
\]
by Proposition \ref{prp:rep_fc_cont}.
 On the other hand, for all $f \in C_c(\tilde G)$ and $Hx \in \tilde G$,
\[\begin{split}
\pi(\Kern_{F(\opL)}) f(Hx) 
&= \int_G \Kern_{F(\opL)}(y) (\pi(y^{-1}) f)(Hx) \,d\meas(y) \\
&= \int_G \Kern_{F(\opL)}(y) f(Hxy^{-1}) \,d\meas(y) \\
&= \int_{\tilde G} P(\Kern_{F(\opL)} g)(Hy) \,d\tilde\meas(Hy),
\end{split}\]
where
\[
g(y) = f(Hxy^{-1}) = f(Hx(Hy)^{-1}) = \tilde g(Hy).
\]
Hence $P(\Kern_{F(\opL)} g) = P(\Kern_{F(\opL)}) \tilde g$, and
\[\begin{split}
\pi(\Kern_{F(\opL)}) f(Hx) 
&= \int_{\tilde G} f(Hx (Hy)^{-1}) P(\Kern_{F(\opL)})(Hy) \,d\tilde\meas(Hy) \\
&= f*P(\Kern_{F(\opL)})(Hx).
\end{split}\]
This shows that $P \Kern_{F(\opL)}$ is the convolution kernel of $\pi(\Kern_{F(\opL)}) = F(\tilde\opL)$.
\end{proof}

Lemma \ref{lem:quoziente} shows that the proof of statement \eqref{st:szweak} for a certain group $G$ can be reduced to the proof of the analogous statement for some quotient of $G$.
Hence it is enough to prove Theorem \ref{thm:mainmain} in the particular cases where $G = \RR$ or $G = \Mot$. The proof is given in Section \ref{s:se2} below. First, however, we show that Corollary \ref{cor:main} is a particular case of Theorem \ref{thm:mainmain} (cf.\ also \cite[Remark 3.9]{breuillard_geometry_2014}).

\begin{prp}\label{prp:quozienti_particolari}
Assume that $G$ is solvable and noncompact. Then there is a closed subgroup $H$ of $G$ such that the quotient $G/H$ is isomorphic to either $\RR$ or $\Mot$.
\end{prp}
\begin{proof}
Let $N$ be the nilradical of $G$. Then $N$ is a closed connected normal subgroup of $G$ \cite[Theorem 3.18.13]{varadarajan_lie_1974} and the quotient $G/N$ is abelian, because $G$ is solvable. Therefore $G/N$ it is isomorphic to $\RR^k \times \TT^h$ for some $k,h \in \NN$. If $k > 0$, then there is a quotient of $G$ that is isomorphic to $\RR$.

Suppose instead that $k=0$, i.e., $G/N$ is isomorphic to $\TT^h$. Since $G$ is noncompact, $N$ is noncompact as well. Let $\{e\} = Z_0 \subseteq Z_1 \subseteq \dots$ denote the ascending central series of $N$; then the $Z_j$ are closed characteristic subgroups of $N$ (hence they are closed normal subgroups of $G$) and moreover, since $N$ is nilpotent and nontrivial, there exists $s \in \NN$ such that $Z_s \neq N$ but $Z_{s+1} = N$. Note also that $N/Z_s$ is noncompact; this is clear when $s=0$, since $N$ is noncompact, while if $s>0$ then $N/Z_s$ is isomorphic to the quotient $\tilde N/ \tilde Z_s$  of the universal covering group of $N$ by the $s$th element of the respective ascending central series (indeed $N$ is a quotient of $\tilde N$ by a discrete central subgroup, and this subgroup is contained in $\tilde Z_s$ when $s>0$). Hence, modulo replacing $G$ with $G/Z_s$, we may assume that $N$ is abelian.

In particular, $N$ is isomorphic to $\TT^m \times \RR^n$, and actually the compact subgroup $K$ of $N$ corresponding to $\TT^m \times \{0\}$ is a characteristic topological subgroup of $N$ (it is the set of the elements of $N$ whose powers do not escape to infinity). Hence $K$ is a normal compact subgroup of $G$ and, modulo replacing $G$ with $G/K$, we may assume that $m = 0$, i.e., $N$ is isomorphic to $\RR^n$. Note that $n>0$, since $N$ is noncompact.

By \cite[Chapter III, Theorem 2.3]{hochschild_structure_1965}, $N$ is a semidirect factor of $G$, i.e., $G = HN$ for some compact subgroup $H$ of $G$ isomorphic to $G/N$. Correspondingly $\lie{g}$ decomposes as (linear) direct sum of the Lie subalgebra $\lie{h}$ and the ideal $\lie{n}$.

The adjoint action of $G$ on $\lie{g}$ induces an action of $H$ on $\lie{n}$. Since $H \cong \TT^h$, we can find an inner product on $\lie{n}$ such that $H$ acts on $\lie{n}$ by isometries and decompose $\lie{n}$ into a direct sum of irreducible $H$-invariant subspaces of dimension $1$ or $2$. 
If all these $H$-invariant subspaces are $1$-dimensional, then the action of $H$ on $\lie{n}$ is trivial, i.e., the action of $H$ on $N$ by conjugation is trivial; in this case, $H$ is a normal compact subgroup of $G$ and $G/H \cong N \cong \RR^n$, so $G$ has a quotient isomorphic to $\RR$ in this case.

Suppose instead that in the decomposition of $\lie{n}$ into irreducible $H$-invariant subspaces a $2$-dimensional subspace occurs. Let $\lie{n}'$ be the sum of all the other irreducible $H$-invariant subspaces in the decomposition, and let $N'$ be the corresponding subgroup of $N$. Then $N'$ is a normal subgroup of $G$, since its Lie algebra $\lie{n}'$ is $H$-invariant by construction, and it is $N$-invariant since $N$ is abelian. Modulo replacing $G$ with $G/N'$, we may assume that $N$ is isomorphic to $\RR^2$ and $H$ acts irreducibly on $\lie{n}$ by isometries.

The kernel $L$ of the action of $H$ on $\lie{n}$ is a closed subgroup of $H$, hence it is a compact subgroup of $G$. Moreover $L$ is central in $G$, because each elements of $L$ commutes with the elements of both $H$ and $N$. In particular $L$ is a normal closed subgroup of $G$ and, modulo replacing $G$ with $G/L$, we may assume that the action of $H$ on $\lie{n}$ is faithful. This forces $h=1$, i.e., $G \cong \RR^2 \rtimes \TT$.
\end{proof}

\section{\texorpdfstring{Sub-Laplacians and Schwartz class on $\RR$ and $\Mot$}{Sub-Laplacians and Schwartz class on R and SE(2)}}\label{s:se2}

The validity of statement \eqref{st:szweak} in the case $G = \RR$ is trivial and well known.
Indeed any sub-Laplacian $\opL$ on $\RR$ can be written, up to rescaling, as $\opL = -\partial_t^2$.
Then $F(\tau^2) = \hat \Kern_{F(\opL)}(\tau)$ for all $\tau \in \RR$, where $\hat f$ denotes the Fourier transform of $f : \RR \to \CC$. The result is then essentially reduced to the fact that the Euclidean Fourier transform preserves the Schwartz class. 
Perhaps the only subtlety in the argument is making sure that the change of variables $\lambda = \tau^2$ preserves the Schwartz class; this is done by means of the following classical result, essentially due to Whitney \cite{whitney_differentiable_1943}, that we record here for future convenience.

\begin{prp}\label{prp:even_schwartz}
Let $K \in \Sz(\RR)$ be even. Then there exists $\tilde K \in \Sz(\RR)$ such that $\tilde K(\tau^2) = K(\tau)$ for all $\tau \in \RR$.
\end{prp}

Hence, to complete the proof of Theorem \ref{thm:mainmain}, we are left with the case $G = \Mot$. Recall that the group $\Mot$ is the semidirect product $\RR^2 \rtimes \TT$, where $\TT$ acts on $\RR^2$ by rotations. Up to identifying $\RR^2$ with $\CC$ and $\TT$ with the unit circle in $\CC$, we can write the group law as
\[
(z,e^{i\theta}) \cdot (z',e^{i\theta'}) = (z+e^{i\theta} z',e^{i(\theta+\theta')})
\]
for all $(z,e^{i\theta}),(z',e^{i\theta'}) \in \CC \times \TT$. It is immediately verified that the product of the Lebesgue measure on $\CC$ and the normalized Haar measure on $\TT$ is a Haar measure on $\Mot$. Moreover, if we write $z = x+iy$, the basis $X,Y,T$ of the Lie algebra $\mot$ of $\Mot$ obtained by extending $\partial_x,\partial_y,\partial_\theta$ to left-invariant vector fields satisfies the commutation relations
\begin{equation}\label{eq:motcr}
[T,X] = Y, \qquad [T,Y] = -X, \qquad [X,Y] = 0.
\end{equation}

\begin{prp}
In the above coordinates, the Schwartz class $\Sz(\Mot)$ on the group $\Mot$ coincides as a Fr\'echet space with the Schwartz class $\Sz(\CC \times \TT)$ on the abelian group $\CC \times \TT$.
\end{prp}
\begin{proof}
Let $U$ be the closed unit disc in $\CC$. Then $V = U \times \TT$ is a compact neighbourhood of the identity in $\Mot$, and
\[
V^k = (kU) \times \TT.
\]
Hence, for all $(z,e^{i\theta}) \in \Mot$,
\[
\tau_V(z,e^{i\theta}) = \min\{ k \in \NN \tc (z,e^{i\theta}) \in V^k \} \sim 1+|z|
\]

In addition, if we write $Z = (X-iY)/2$ and $\bar Z = (X+iY)/2$, then
\[
T  =\partial_\theta, \qquad Z = e^{i\theta} \partial_{z}, \qquad \bar Z = e^{-i\theta} \partial_{\bar z},
\]
where $\partial_{z} = (\partial_x - i \partial_y)/2$ and $\partial_{\bar z} = (\partial_x + i \partial_y)/2$, and consequently
\[
Z^p \bar Z^q T^r = e^{i(p-q) \theta} \partial_{z}^p \partial_{\bar z}^q \partial_\theta^r
\]
for all $p,q,r \in \NN$.

By combining this information, it is easily shown that every Schwartz seminorm on $\Mot$ is controlled by some Schwartz seminorm on $\CC \times \TT$, and vice versa.
\end{proof}

\begin{lem}\label{lem:mot_automorphism}
Every automorphism of the Lie algebra $\mot$ of $\Mot$ is the differential of a Lie group automorphism of $\Mot$.
\end{lem}
\begin{proof}
The universal covering group of $\Mot$ is $\RR^2 \rtimes \RR$, with group law
\[
(z,t) \cdot (z',t') = (z+e^{it} z',t+t'),
\]
and the kernel of the covering map $(z,t) \mapsto (z,e^{it})$ is the centre of $\RR^2 \rtimes \RR$, which is a characteristic subgroup of $\RR^2 \rtimes \RR$. Consequently, every automorphism of the universal covering group $\RR^2 \rtimes \RR$ descends to an automorphism of $\Mot$, and every automorphism of the Lie algebra $\mot$ of $\Mot$ corresponds to a Lie group automorphism of $\Mot$.
\end{proof}

The following result shows that sub-Laplacians on $\Mot$ have a ``normal form'' (see \cite{baudoin_subelliptic_2015} for similar results).

\begin{prp}\label{prp:classif_sublap_mot}
Let $\opL$ be a sub-Laplacian on $\Mot$. Then there exists a basis $X,Y,T$ of the Lie algebra of $\Mot$ satisfying the commutation relations \eqref{eq:motcr} and such that either
\begin{equation}\label{eq:sub1}
-\alpha \opL = T^2 + Y^2 + \beta (X^2+Y^2)
\end{equation}
for some $\alpha > 0$ and $\beta \geq 0$, or
\begin{equation}\label{eq:sub2}
-\alpha \opL = T^2 + X^2+Y^2
\end{equation}
for some $\alpha > 0$.
\end{prp}
\begin{proof}
Let $X,Y,T$ be the basis of $\mot$ extending $\partial_x,\partial_y,\partial_t$ as above. An arbitrary sub-Laplacian $\opL$ on $SE(2)$ has the form $\opL = -\sum_j X_j^2$, where $X_j = a_j T + U_j$ for some $a_j \in \RR$ and $U_j \in \Span\{X,Y\}$. Note that not all $a_j$ are zero (otherwise $\opL$ would not be hypoelliptic), so, up to rescaling $\opL$, we may also assume that $\sum_j a_j^2 = 1$. Hence
\[
-\opL = \sum_j (a_j T + U_j)^2 = T^2 + \sum_j U_j^2 + T U + U T = (T+U)^2 + \sum_j U_j^2 - U^2
\]
where $U = \sum_j a_j U_j \in \Span\{X,Y\}$.

Note that $\sum_j U_j^2 - U^2 = (\sum_j a_j^2) (\sum_j U_j^2) - (\sum_j a_j U_j)^2$ can be thought of as a quadratic form on the dual of $\Span\{X,Y\}$, which is positive semidefinite by the Cauchy--Schwarz inequality. Hence, by the spectral theorem, there exists $s \in \RR$ such that, if $\tilde X = (\cos s) X - (\sin s) Y$ and $\tilde Y = (\sin s) X  + (\cos s) Y$, then $\sum_j U_j^2 - U^2 = b \tilde X^2 + c \tilde Y^2$ for some $b,c \in \Rnon$ with $b \leq c$. If we set $\tilde T = T + U$, then
\[
-\opL = \tilde T^2 + b\tilde X^2 + c \tilde Y^2.
\]
Note that $b,c$ are not both zero (otherwise $\opL$ would not be hypoelliptic).

If $b=c$, then we reduce to the case \eqref{eq:sub2} by choosing the basis $\tilde T,\sqrt{b} \tilde X,\sqrt{b} \tilde Y$ of $\mot$.

If $b<c$, then we can write $-\opL = \tilde T^2 + (c-b)\tilde Y^2 + b(\tilde X^2 + \tilde Y^2)$, and we reduce to the case \eqref{eq:sub1} by choosing the basis $\tilde T,\sqrt{c-b} \tilde X,\sqrt{c-b} \tilde Y$ of $\mot$.
\end{proof}

The following result completes the proof of Theorem \ref{thm:mainmain}.

\begin{prp}\label{prp:se2sz}
Statement \eqref{st:szweak} holds for any sub-Laplacian $\opL$ on $G = \Mot$.
\end{prp}
\begin{proof}
Let $\alpha \in \Rpos$, $\beta \in \Rnon$ and $X,Y,T$ be the basis of $\mot$ given by Proposition \ref{prp:classif_sublap_mot} corresponding to a given sub-Laplacian $\opL$ on $\Mot$, and define
\[
\Delta = -(T^2+X^2+Y^2), \qquad \opL_0 = -(T^2 + Y^2), \qquad \Delta_0 = -(X^2+Y^2).
\]
Note that, to the purpose of proving statement \eqref{st:szweak}, by rescaling $\opL$ we may assume that $\alpha=1$. Moreover, by Lemma \ref{lem:mot_automorphism}, up to an automorphism of $\Mot$ we may assume that $X,Y,T$ is the ``standard basis'' extending $\partial_x,\partial_y,\partial_\theta$; indeed, note that automorphisms of $\Mot$ preserve the Schwartz class $\Sz(\Mot)$.

It is easily seen that
\[
\Delta = -(\partial_\theta^2+\partial_x^2+\partial_y^2).
\]
In particular, if $\opL = \Delta$, then $\opL$ coincides with a translation-invariant Laplacian on the abelian group $\CC \times \TT$, and statement \eqref{st:szweak} follows by what has already been proved (note that $\CC \times \TT$ has a quotient isomorphic to $\RR$). So it remains only to consider the case $\opL = \opL_0 + \beta \Delta_0$.

Let $\TT$ be given the normalised Haar measure. A family of irreducible unitary representations $\pi_r$ of $\Mot$ on $L^2(\TT)$, where $r \in \Rpos$, is given as follows:
\[
\pi_r(z,e^{i\theta}) f(e^{i\phi})  = e^{i r\Re(z e^{i\phi})} f(e^{i(\theta+\phi)})
\]
(cf., e.g., \cite[eq.\ (3)]{dooley_extension_1984}). Correspondingly, for all $K \in L^1(\Mot)$ and $f \in L^2(\TT)$,
\begin{equation}\label{eq:motrepn}\begin{split}
\pi_r(K) f(e^{i\phi}) 
&= \int_\Mot K(z,e^{i\theta}) \,  \pi_r(-e^{-i\theta}z,e^{-i\theta}) \, f(e^{i\phi}) \,dz \,d\theta \\
&= \int_\Mot K(z,e^{i\theta}) \,  e^{-ir\Re(z e^{i(\phi-\theta)})} \, f(e^{i(\phi-\theta)}) \,dz \,d\theta \\
&= \int_\Mot K(z,e^{i(\phi-\theta)}) \,  e^{-ir\Re(z e^{i\theta})} \, f(e^{i\theta}) \,dz \,d\theta \\
&= \int_\Mot K(z,e^{i(\phi-\theta)}) \,  e^{-i\Re(z \overline{r e^{-i\theta}})} \, f(e^{i\theta}) \,dz \,d\theta \\
&= \int_\TT K(\widehat{re^{-i\theta}},e^{i(\phi-\theta)}) f(e^{i\theta}) \,d\theta,
\end{split}\end{equation}
where $K(\widehat\cdot, \cdot)$ denotes the partial Euclidean Fourier transform of $K$ along $\CC$. Moreover
\[
d\pi_r(X) = i r \cos\phi, \qquad d\pi_r(Y) = -i r \sin\phi, \qquad d\pi_r(T) = \partial_\phi.
\]

Note now that $\opL_0$ and $\Delta_0$ commute; indeed $\Delta_0$ is in the centre of $\Diff(\Mot)$ and
\[
d\pi_r(\Delta_0) = r^2,
\]
while
\[
d\pi_r(\opL_0) = -\partial_\phi^2 + r^2\sin^2\phi = \opML_{r^2},
\]
where $\opML_{q}$ denotes, for all $q \in \RR$, the Mathieu operator of Section \ref{s:mathieu}.

Let $\lambda^q$ and $H^q$ be the minimum eigenvalue and the corresponding normalized eigenfunction of $\opML_q$ (denoted in Section \ref{s:mathieu} as $\lambda_{(0,0),0}^q$ and $H_{(0,0),0}^q$ respectively).
Then, for all continuous functions $F : \Rnon \to \CC$ with $\Kern_{F(\opL)} \in \Sz(\Mot)$, we can write, by \eqref{eq:motrepn} and Proposition \ref{prp:rep_fc_cont},
\[\begin{split}
F(\lambda^{r^2} + \beta r^2) &= \langle F(d\pi_r(\opL)) H^{r^2}, H^{r^2} \rangle \\
&= \langle \pi_r(\Kern_{F(\opL)}) H^{r^2}, H^{r^2} \rangle \\
&= G_F(r),
\end{split}\]
where
\begin{equation}\label{eq:kernel}
G_F(r) = \int_\TT \int_\TT \Kern_{F(\opL)}(\widehat{re^{-i\theta}},e^{i(\phi-\theta)}) \, H^{r^2}(e^{i\theta}) \, H^{r^2}(e^{i\phi}) \,d\theta \,d\phi.
\end{equation}

Note that the above formula defines $G_F$ for all $r \in \RR$. Moreover, since $\Sz(\Mot) = \Sz(\CC \times \TT) \cong \Sz(\CC) \mathop{\hat\otimes} C^\infty(\TT)$ and the Euclidean Fourier transform preserves the Schwartz class, from Propositions \ref{prp:eigen_smooth} and \ref{prp:eigen_estimates_poly} it is easily seen that $G_F \in \Sz(\RR)$. In addition
\begin{equation}\label{eq:kernel_parity}
G_F(-r) 
= \int_\TT \int_\TT \Kern_{F(\opL)}(\widehat{re^{-i(\theta+\pi)}},e^{i(\phi-\theta)}) \, H^{r^2}(e^{i\theta}) \, H^{r^2}(e^{i\phi}) \,d\theta \,d\phi = G_F(r),
\end{equation}
since $H^{r^2}$ is $\pi$-periodic. Hence $G_F(r) = \tilde G_F(r^2)$ for some $\tilde G_F \in \Sz(\RR)$ by Proposition \ref{prp:even_schwartz} and $F(\lambda^q + \beta q) = \tilde G_F(q)$ for all $q \in \Rnon$.

Let $\psi(q) = \lambda^q + \beta q$. By Propositions \ref{prp:eigen_continuous}, \ref{prp:eigen_smooth} and \ref{prp:eigen_asymp}, $\psi : \RR \to \RR$ is smooth, bijective, strictly increasing and $\psi(0) = 0$. In particular $\tilde G_F \circ \psi^{-1}$ is smooth; what we need to check is that its restriction to $\Rnon$ is in $\Sz(\Rnon)$.

On the other hand, for all $k \in \NN$,
\[
(\tilde G_F \circ \psi^{-1})^{(k)}(t) = \sum_{j=0}^k (\tilde G_F)^{(j)} (\psi^{-1}(t)) \, P_{k,j}((\psi^{-1})'(t),\dots,(\psi^{-1})^{(k)}(t))
\]
for certain polynomials $P_{k,j}$ (where $P_{0,0} = 1$ and $P_{k,0} = 0$ for all $k>0$),
while, for all $k > 0$,
\[
(\psi^{-1})^{(k)}(t) = \frac{Q_k(\psi'(\psi^{-1}(t)),\dots,\psi^{(k)}(\psi^{-1}(t)))}{ (\psi'(\psi^{-1}(t))^{2k-1} }.
\]
for certain polynomials $Q_k$. Consequently
\begin{equation}\label{eq:comp_inv_deriv}
(\tilde G_F \circ \psi^{-1})^{(k)}(t) = R_{F,k}(\psi^{-1}(t)),
\end{equation}
where
\[
R_{F,k}(q) = \sum_{j=0}^k (\tilde G_F)^{(j)} (q) \, \frac{\tilde P_{k,j}(\psi'(q),\dots,\psi^{(k)}(q))}{(\psi'(q))^{a(k,j)}}
\]
for certain polynomials $P_{k,j}$ and exponents $a(k,j) \in \NN$.

Since $\psi(q) = \lambda^q + \beta q$, from Proposition \ref{prp:eigen_estimates_poly} it follows immediately that its derivatives have at most polynomial growth as $q \to \infty$. Moreover $1/\psi'(q) \lesssim (1+q)^{1/2}$; this is clear if $\beta > 0$, because $\psi'(q) \geq \beta$; otherwise, $\psi'(q) = \partial_q \lambda^q \sim \lambda_q/q \sim q^{1/2}$ as $q \to \infty$ by Proposition \ref{prp:eigen_asymp}. Since $\tilde G_F \in \Sz(G)$, it follows immediately that $R_{F,k}$ decays faster than any polynomial on $\Rnon$.

On the other hand, by Proposition \ref{prp:eigen_asymp},
\[
1 + \psi(q) \sim 
\begin{cases}
1 + q &\text{if $\beta > 0$,}\\
(1 + q)^{1/2} &\text{if $\beta = 0$,}
\end{cases}
\]
that is
\[
1 + t \sim 
\begin{cases}
1 + \psi^{-1}(t) &\text{if $\beta > 0$,}\\
(1 + \psi^{-1}(t))^{1/2} &\text{if $\beta = 0$,}
\end{cases}
\]
hence $R_{F,k} \circ \psi^{-1}$ also decays faster than any polynomial on $\Rnon$. From \eqref{eq:comp_inv_deriv} we can conclude that $\tilde G_F|_{\Rnon} \in \Sz(\Rnon)$.
\end{proof}

\section{The kernel transform and its adjoint}\label{s:kerneltransform}

From the boundedness properties of the operator $T_\opL : F \mapsto K_{F(\opL)}$ and its adjoint $T_\opL^*$, it follows that $T_\opL$ is an integral operator $\chi_\opL$, whose integral kernel satisfies a number of properties summarized below (see \cite{tolomeo_tesilaurea_2015} for a more detailed study).

\begin{prp}\label{prp:kerneltransformkernel}
There exists a unique $\chi_\opL \in L^\infty(\Rnon \times G,\Pmeas_\opL \times \meas)$ such that, for all $F \in L^1 \cap L^\infty(\Rnon,\Pmeas_\opL)$,
\begin{equation}\label{eq:opLtransform}
\Kern_{F(\opL)}(x) = \int_\Rnon F(\lambda) \, \chi_\opL(\lambda,x)\,d\Pmeas_\opL(\lambda).
\end{equation}
Moreover the kernel $\chi_\opL$ satisfies the following properties.
\begin{enumerate}[label=\textup{(\roman*)}]
\item\label{en:ktk_real} $\chi_\opL$ is real-valued.
\item\label{en:ktk_adj} For all $k \in L^1(G)$,
\[
T_\opL^* k(\lambda) = \int_G k(x) \, \chi_\opL(\lambda,x) \,d\meas(x) \qquad\text{for }\Pmeas_\opL\text{-a.e.\ }\lambda.
\]
\item\label{en:ktk_inv} \emph{(Inversion formula for the kernel transform)} For all $F \in L^\infty(\Rnon,\Pmeas_\opL)$ such that $K_{F(\opL)} \in L^1(G)$,
\begin{equation}\label{eq:opLinversion}
F(\lambda) = \int_G \Kern_{F(\opL)}(x) \, \chi_\opL(\lambda,x)\,d\meas(x) \qquad\text{for }\Pmeas_\opL\text{-a.e.\ }\lambda.
\end{equation}
\item\label{en:ktk_smooth} For $\Pmeas_\opL$-almost all $\lambda \in \Rnon$, $\chi_\opL(\lambda,\cdot) \in C^\infty(G)$ and $\opL \chi_\opL(\lambda,\cdot) = \lambda \chi_\opL(\lambda,\cdot)$.
\end{enumerate}
\end{prp}
\begin{proof}
From the boundedness of the operator  $T_\opL:L^1(\Rnon,\Pmeas_\opL)\to C_0(G)$, it follows that $T_\opL$ is an integral operator with (uniquely determined) kernel $\chi_\opL \in L^\infty(\Rnon \times G,\Pmeas_\opL \times \meas)$, whence \eqref{eq:opLtransform}.

Nonnegativity of the heat kernel $\Kern_{\exp(-t\opL)}$ shows that $\chi_\opL$ is real-valued, i.e., part \ref{en:ktk_real}. In particular, $\chi_\opL$ serves also as integral kernel of the adjoint operator $T_\opL^*$ (swapping the roles of the variables), whence part \ref{en:ktk_adj}.

From the Plancherel identity \eqref{eq:opLplancherel} we deduce that $T_\opL^* T_\opL$ is the identity operator on $L^2(\Rnon,\Pmeas_\opL)$, whence the inversion formula \eqref{eq:opLinversion} follows whenever $\Kern_{F(\opL)} \in L^1 \cap L^2(G)$. By using the heat kernel as approximate identity, the inversion formula extends to all $F \in L^\infty(\Rnon,\Pmeas_\opL)$ with $K_{F(\opL)} \in L^1(G)$.

Finally, from the algebra morphism property
$\Kern_{F(\opL)}\ast\Kern_{G(\opL)} = \Kern_{FG(\opL)}$
and Lebesgue differentiation
one gets that 
\[
\chi_\opL(\lambda,\cdot)\ast \Kern_{F(\opL)} = \Kern_{F(\opL)} \ast \chi_\opL(\lambda,\cdot) = F(\lambda) \, \chi_\opL(\lambda,\cdot) \qquad\text{for }\Pmeas_\opL\text{-a.e.\ }\lambda,
\]
from which part \ref{en:ktk_smooth} follows easily (take, e.g., $F(\lambda) = e^{-\lambda}$ and use the smoothness of the heat kernel).
\end{proof}

From the above inversion formula \eqref{eq:opLinversion} and dominated convergence it is clear that continuity of the kernel $\chi_\opL$ in the variable $\lambda$ (i.e., the fact that $\chi_\opL(\cdot,x)$ is continuous on $\Sigma_\opL$ for $\meas$-almost all $x \in G$) would imply, together with \eqref{eq:vanishinfty}, statement \eqref{st:rl} for any Lie group of polynomial growth: actually, it would imply the stronger statement that $T_\opL^*$ maps $L^1(G)$ into $C(\Rnon)$. However, differently from continuity of $\chi_\opL$ in the variable $x$ (which always holds by Proposition \ref{prp:kerneltransformkernel}\ref{en:ktk_smooth}), continuity in the variable $\lambda$ is a much subtler property, as we shall see.

One case where continuity of $\chi_\opL$ in the variable $\lambda$ can be proved is that of homogeneous sub-Laplacians on stratified groups (see, e.g., \cite{folland_hardy_1982} for basic definitions and results), thus proving Theorem \ref{thm:riemannlebesgue}\ref{en:riemannlebesguehomogeneous}. We remark that the characterisation of the Plancherel measure for homogeneous sub-Laplacians given below is well-known \cite{christ_multipliers_1991,de_michele_multipliers_1987,hulanicki_almost_1983}.

\begin{prp}
Assume that $G$ is a stratified group, with homogeneous dimension $Q$ and automorphic dilations $\delta_t$, and $\opL$ is a homogeneous sub-Laplacian thereon. Then
\begin{equation}\label{eq:Pmeas_hom}
d\Pmeas_\opL(\lambda) = c \lambda^{Q/2-1} d\lambda
\end{equation}
for some $c \in \Rpos$, and
\begin{equation}\label{eq:chi_hom}
\chi_\opL(t^{2} \lambda,x) = \chi_\opL(\lambda,\delta_{t}(x))
\end{equation}
for almost all $\lambda \in \Rnon$, $x \in G$ and $t \in \Rpos$. In particular $\Sigma_\opL = \Rnon$,  $\chi_\opL \in C(\Rnon \times G)$, and statement \eqref{st:rl} holds in this case.
\end{prp}
\begin{proof}
Formulas \eqref{eq:Pmeas_hom} and \eqref{eq:chi_hom} are immediate consequences of the relation
\[
\Kern_{F(t^2 \opL)} = t^{-Q} \Kern_{F(\opL)} \circ \delta_{t^{-1}},
\]
due to homogeneity. From \eqref{eq:Pmeas_hom} it is clear that the support of $\Pmeas_\opL$ is the whole $\Rnon$, while \eqref{eq:chi_hom} allows one to deduce joint continuity of $\chi_\opL$ from its continuity in the variable $x$, due to Proposition \ref{prp:kerneltransformkernel}\ref{en:ktk_smooth}. By the above discussion, continuity of $\chi_\opL$ implies statement \eqref{st:rl}.
\end{proof}

However, continuity of $\chi_\opL$ in the variable $\lambda$ is not true in general, nor is it necessary for the validity of statement \eqref{st:rl}, as the following result for \emph{abelian} groups (proving Theorem \ref{thm:riemannlebesgue}\ref{en:riemannlebesgueabelian}) shows.

\begin{prp}\label{prp:abelianRL}
Let $G$ be an abelian connected Lie group and $\opL$ a sub-Laplacian thereon. Then statement \eqref{st:rl} holds. However, if $G$ is isomorphic to $\RR \times \TT$, 
then 
the kernel $\chi_\opL$ is not continuous in the variable $\lambda$.
\end{prp}
\begin{proof}
Since $G$ is an abelian connected Lie group, up to an isomorphism we may assume that $G = \RR^n \times \TT^m$. Moreover, up to an automorphism and rescaling, we may also assume that $\opL = -\nabla_x \cdot \nabla_x - (A \nabla_y) \cdot (A \nabla_y)$ on $\RR^n_x \times \TT^m_y$, where $A$ is an invertible real $m \times m$ matrix.

Let $\Four$ be the Fourier transform on $G$, given by 
\[
\Four(\phi)(\xi,k) 
 = \int_{\TT^m} \int_{\RR^n} \phi(x,y) \, e^{-i(\xi \cdot x +k \cdot y)} \,dx\,dy
\]
for all $(\xi,k) \in \RR^n \times \ZZ^m$, and note that $\Four( \opL \phi) = M \Four \phi$, where $M(\xi,k) = |\xi|^2+|Ak|^2$.

In the case $n=0$, this shows that the spectrum $\Sigma_\opL = \{ |Ak|^2 \tc k \in \ZZ^m \}$ of $\opL$ is discrete in $\Rnon$, so any function on $\Sigma_\opL$ vanishing at infinity extends to an element of $C_0(\Rnon)$ and, in view of \eqref{eq:vanishinfty}, statement \eqref{st:rl} is trivially true in this case.

Suppose instead that $n > 0$. Then we can write
\begin{equation}\label{eq:opLtransformRxT2}
\begin{split}
\Kern_{F(\opL)}(x,y) &= \Four^{-1}(F \circ M)(x,y) \\
&= \frac{1}{(2\pi)^{n+m}} \int_{\RR^n} \sum_{k\in\ZZ^m} F(|\xi|^2+|Ak|^2) \, e^{i(\xi \cdot x +k \cdot y)} \,d\xi.
\end{split}
\end{equation}
This, together with the Plancherel theorem for $\Four$, implies that
\begin{equation}\label{eq:RTPmeas2}
d\Pmeas_\opL(\lambda) = \kappa_{n,m} \sum_{\substack{k \in \ZZ^m \\ |Ak|<\sqrt\lambda}} (\lambda-|Ak|^2)^{n/2-1}\, d\lambda,
\end{equation}
for some $\kappa_{n,m} \in \Rpos$.

Let now $F : \Rnon \to \CC$ be a bounded Borel function such that $\Kern_{F(\opL)} \in L^1(G)$. From \eqref{eq:opLtransformRxT2} we deduce that
\[
F(|\xi|^2 + |Ak|^2) = \Four(\Kern_{F(\opL)})(\xi,k)
\]
for almost all $(\xi,k) \in \RR^n \times \ZZ^m$.  Therefore, for almost all $\omega \in S^{n-1}$, the identity
\begin{equation}\label{eq:RTrl2}
F(\lambda) = \Four(\Kern_{F(\opL)})(\sqrt\lambda \omega,0)
\end{equation}
holds for Lebesgue-almost all $\lambda \in \Rpos$. Let us fix such $\omega \in S^{n-1}$; then the right-hand side of \eqref{eq:RTrl2} is a continuous function of $\lambda$, because $\Kern_{F(\opL)} \in L^1(G)$.
Since, by \eqref{eq:RTPmeas2}, $\Pmeas_\opL$ is absolutely continuous with respect to the Lebesgue measure, the identity \eqref{eq:RTrl2} also holds for $\Pmeas_\opL$-almost all $\lambda \in \Rpos$, whence $F$ coincides $\Pmeas_\opL$-almost everywhere with a continuous function on $\Rnon$. This proves statement \eqref{st:rl}.

Assume now that $m=n=1$. Up to an automorphism and rescaling, we may also assume that $A = 1$. From \eqref{eq:opLtransformRxT2} and \eqref{eq:RTPmeas2} we deduce that
\[
\chi_\opL(\lambda,(x,y)) = \frac{\sum_{|k|<\sqrt\lambda} \cos((\lambda-k^2)^{1/2} x) \, e^{iky} \, (\lambda-k^2)^{-1/2}}
{\sum_{ |k|<\sqrt\lambda} (\lambda-k^2)^{-1/2}}.
\]
It is easy to check that, for every $h\in\NN\setminus\{0\}$,
\[
\lim_{\lambda \uparrow {h^2}} \chi_\opL(\lambda,(x,y)) = \frac{\sum_{|k|<h} \cos((h^2-k^2)^{1/2} x) \, e^{iky} \, (h^2-k^2)^{-1/2}}
{\sum_{|k|<h} (h^2-k^2)^{-1/2}} = \sum_{|k|<h} c_k(x) \, e^{iky},
\]
while
\[
\lim_{\lambda \downarrow {h^2}} \chi_\opL(\lambda,(x,y)) = e^{ihy}+e^{-ihy} = 2\cos(hy),
\]
therefore the two one-sided limits are different as functions of $(x,y)$: they are orthogonal in $L^2(\TT)$ as functions of $y$.
\end{proof}

\section{Type I groups and group Plancherel formula}\label{s:groupplancherel}

A fundamental ingredient in the proofs of Sections \ref{s:quotients} and \ref{s:se2} is Proposition \ref{prp:rep_fc_cont}, which allows us to relate the functional calculus of a sub-Laplacian $\opL$ and that of the corresponding operator in any unitary representation, provided the function $F$ defining the operators is continuous.

The continuity hypothesis on $F$ in Proposition \ref{prp:rep_fc_cont} cannot in general be removed. However one may wonder whether, for a given bounded Borel function $F : \RR \to \CC$, one might obtain the validity of \eqref{eq:rep_fc_cont} after replacing $F$ with another representative in the equivalence class modulo $\Pmeas_\opL$. The following result shows that, in that case, the new representative must be continuous on the spectrum of $\opL$.

\begin{prp}\label{prp:almost_continuous}
Let $G$ be a connected Lie group of polynomial growth and $\opL$ be a sub-Laplacian thereon.
If $F : \Rnon \to \CC$ is a bounded Borel function such that $\Kern_{F(\opL)} \in L^1(G)$ and \eqref{eq:rep_fc_cont} holds for all $\pi \in \widehat G$, then $F$ is continuous on $\Sigma_\opL$, and there exists a continuous function $\tilde F : \Rnon \to \CC$ such that $F = \tilde F$ $\Pmeas_\opL$-almost everywhere.
\end{prp}
\begin{proof}
As in \cite[Section 4]{martini_spectral_2011}, let $\Eigen_\opL$ be the set of eigenfunctions $\phi$ of $\opL$ of positive type with $\phi(e) = 1$, and let $\Val_\opL : \Eigen_\opL \to \Rnon$ be the map that associates to each eigenfunction the corresponding eigenvalue. By \cite[Proposition 4.5 and Corollary 4.10]{martini_spectral_2011}, if $\Eigen_\opL$ is given the subspace topology induced by the weak-$*$ topology of $L^\infty(G)$ (equivalently, the compact-open topology of $C(G)$ or the Fr\'echet topology of $C^\infty(G)$), $\Val_\opL$ is a continuous, proper and closed map and $\Val_\opL(\Eigen_\opL) = \Sigma_\opL$.

Let $\widehat G$ be the collection of all irreducible unitary representations of $G$ (up to equivalence). For all $\pi \in \widehat G$, let $\Eigen_\opL^\pi$ be the set of all the $\phi \in \Eigen_\opL$ which are diagonal coefficients of $\pi$. In other words, for all $\phi \in \Eigen_\opL^\pi$, there exists a smooth unit vector $v_\phi \in \Hilb^\pi$ such that $\phi = \langle \pi(\cdot) v_\phi,v_\phi \rangle$ and $d\pi(\opL) v_\phi = \Val_\opL(\phi) v_\phi$ \cite[Proposition 4.3]{martini_spectral_2011}. Note moreover that, if $\Eigen_{\opL}^\irr = \bigcup_{\pi \in \widehat G} \Eigen_\opL^\pi$, then, for all $\lambda \in \Rnon$, $\Eigen_{\opL}^\irr \cap \Val_\opL^{-1}(\{\lambda\})$ is the set of extreme points of the convex set $\Val_\opL^{-1}(\{\lambda\}) \subseteq \Eigen_\opL$, which is weakly-$*$ compact in $L^\infty(G)$ \cite[Proposition 4.6]{martini_spectral_2011}; in particular, $\Val_\opL(\Eigen_\opL^\irr) = \Val_\opL(\Eigen_\opL) = \Sigma_\opL$.

Let $F : \Rnon \to \CC$ be a bounded Borel function such that $\Kern_{F(\opL)} \in L^1(G)$ and \eqref{eq:rep_fc_cont} holds for all $\pi \in \widehat G$. 
Then, for all $\pi \in \widehat G$ and $\phi \in \Eigen_\opL^\pi$,
\begin{equation}\label{eq:positivetype_identity}
\langle \Kern_{\opL}, \phi \rangle = \langle \pi(\Kern_{\opL}) v^\phi, v^\phi \rangle = \langle F(d\pi(\opL)) v^\phi, v^\phi \rangle = F(\Val_\opL(\phi)).
\end{equation}
In particular, for all $\lambda \in \Rnon$, the function $\phi \mapsto \langle \Kern_{\opL}, \phi \rangle$ is constant on $\Eigen_\opL^\irr \cap \Val_\opL^{-1}(\{\lambda\})$, hence it is constant on its closed convex hull $\Val_\opL^{-1}(\{\lambda\})$. This shows that the identity
\[
\langle \Kern_{\opL}, \phi \rangle = F(\Val_{\opL}(\phi))
\]
holds for all $\phi \in \Eigen_\opL$. Since $\Kern_{\opL} \in L^1(G)$, the above identity shows that the function $F \circ \Val_{\opL}$ is continuous on $\Eigen_{\opL}$. On the other hand, $\Val_\opL : \Eigen_\opL \to \Rnon$ is a continuous closed map, so the topology of the spectrum $\Sigma_\opL = \Val_\opL(\Eigen_\opL)$ induced by $\Rnon$ is the same as the quotient topology induced by $\Val_\opL$, and therefore $F$ is continuous on $\Sigma_\opL$.

Finally, since $\Sigma_\opL$ is closed in $\Rnon$, the bounded continuous function $F|_{\Sigma_\opL}$ can be extended to a bounded continuous function $\tilde F : \Rnon \to \CC$ and clearly  $F = \tilde F$ $\Pmeas_\opL$-almost everywhere.
\end{proof}

\begin{rem}
The polynomial growth assumption in Proposition \ref{prp:almost_continuous} could be relaxed (the same proof works by just assuming $G$ amenable). In addition, the above result, as many others in this section, can be straightforwardly extended from the case of a sub-Laplacian $\opL$ to the case of a ``weighted subcoercive system'' of commuting differential operators on $G$, as in \cite{martini_spectral_2011}.
\end{rem}

The previous result allows us to obtain an equivalent formulation of statement \eqref{st:rl} i.e., the Riemann--Lebesgue lemma for a sub-Laplacian $\opL$.

\begin{cor}\label{cor:equivalent_rl}
Let $G$ be a connected Lie group of polynomial growth and $\opL$ be a sub-Laplacian thereon. The following are equivalent:
\begin{enumerate}[label=\textup{(\roman*)}]
\item\label{en:rl} Statement \eqref{st:rl} holds for $G$ and $\opL$.
\item\label{en:change_rep} 
Every bounded Borel function $F : \Rnon \to \CC$ with $\Kern_{F(\opL)} \in L^1(G)$ coincides $\Pmeas_\opL$-almost everywhere with another bounded Borel function $\tilde F : \Rnon \to \CC$ such that
\begin{equation}\label{eq:rep_fc_mod}
\pi(K_{F(\opL)}) = \tilde F(d\pi(\opL))
\end{equation}
for all $\pi \in \widehat G$.
\end{enumerate}
\end{cor}
\begin{proof}
If \ref{en:rl} holds, for all bounded Borel functions $F : \Rnon \to \CC$ such that $\Kern_{F(\opL)} \in L^1(G)$, we can take the continuous representative $\tilde F$ modulo $\Pmeas_\opL$ given by statement \eqref{st:rl} and apply Proposition \ref{prp:rep_fc_cont} to $\tilde F$ to obtain \eqref{eq:rep_fc_mod}.

Conversely, if \ref{en:change_rep} holds, for all bounded Borel functions $F : \Rnon \to \CC$ such that $\Kern_{F(\opL)} \in L^1(G)$ we can take the representative $\tilde F$ modulo $\Pmeas_\opL$ given by \ref{en:change_rep} and obtain, by Proposition \ref{prp:almost_continuous}, that $F = \tilde F = \tilde{\tilde F}$ $\Pmeas_\opL$-almost everywhere for some continuous function $\tilde{\tilde F} : \Rnon \to \CC$.
\end{proof}

Assume now that the Lie group of polynomial growth $G$ is type I. Then a Plancherel formula for the group Fourier transform holds. In other words, if $\widehat G$ is the dual object of $G$ (that is, the collection of irreducible unitary representations of $G$ modulo equivalence), then there exists a measurable structure on $\widehat G$ so that $\pi \mapsto \Hilb^\pi$ and $\pi \mapsto \pi$ are respectively a measurable field of Hilbert spaces and a measurable fields of unitary representations thereon; in addition, there exists a unique measure $\GPmeas$ on $\widehat G$, called Plancherel measure, such that the correspondence
\[
\Four : f \mapsto (\pi(f))_{\pi \in \widehat G},
\]
initially defined on $L^1 \cap L^2(G)$, extends to an isometric isomorphism from $L^2(G)$ to $\int^\oplus_{\widehat G} \HS(\Hilb^\pi) \,d\GPmeas(\pi)$. In particular
\begin{equation}\label{eq:Gplancherel}
\int_G |f(x)|^2 \,d\meas(x) = \int_{\widehat G} \| \pi(f) \|_{\HS}^2 \,d\GPmeas(\pi)
\end{equation}
for all $f \in L^1 \cap L^2(G)$ (see \cite[\S 7.5]{folland_course_1995} and references therein). 

Let $\opL$ be a sub-Laplacian on $G$. By means of the group Plancherel formula, we can obtain a version of Proposition \ref{prp:rep_fc_cont} where the continuity assumption on $F$ is weakened.

\begin{prp}\label{prp:rep_fc}
Let $F : \Rnon \to \CC$ be a bounded Borel function such that $K_{F(\opL)} \in L^1(G)$. Then
\begin{equation}\label{eq:rep_fc}
\pi(\Kern_{F(\opL)}) = F(d\pi(\opL))
\end{equation}
for $\GPmeas$-a.e.\ $\pi \in \widehat G$.
\end{prp}
\begin{proof}
Recall that $d\pi(\opL)$ is (essentially) self-adjoint and nonnegative on $\Hilb^\pi$ for all $\pi \in \widehat G$, and moreover, if $E = \exp(-\opL)$, then $\Kern_E \in \Sz(G)$ and $\pi(\Kern_{E}) = \exp(-d\pi(\opL))$ for all $\pi \in \widehat G$ by Proposition \ref{prp:rep_fc_cont}. By spectral mapping, we are then reduced to proving the following result: for every bounded Borel function $F : [0,1] \to \CC$ such that $\Kern_{F(E)} \in L^1(G)$, the identity
\[
\pi(\Kern_{F(E)}) = F(\pi(\Kern_{E}))
\]
holds for $\GPmeas$-almost all $\pi \in \widehat G$.

Note now that, for all $f \in L^1 \cap L^2(G)$ and $\pi \in \widehat G$,
\[
\pi(E f) = \pi(f * \Kern_{E}) = \pi(\Kern_{E}) \pi(f);
\]
in other words, $\Four$ intertwines $E$ and $M = \int^\oplus_{\widehat G} M_\pi \,d\GPmeas(\pi)$, where $M_\pi : \HS(\Hilb^\pi) \to \HS(\Hilb^\pi)$ is the operator of multiplication on the left by $\pi(\Kern_{E})$. By uniqueness of the functional calculus (see, e.g., \cite[Theorem 1.51]{folland_course_1995}), $\Four$ intertwines $F(E)$ and $F(M)$ as well.

We now claim that
\[
F(M) = \int^\oplus_{\widehat G} F(M_\pi) \,d\GPmeas(\pi).
\]
This again follows by uniqueness of the functional calculus. Indeed, the fact that $\pi \mapsto F(M_\pi)$ is a measurable field of operators follows easily from an approximation argument \cite[Lemma 1.50]{folland_course_1995} and the fact that $\pi \mapsto M_\pi$ is. Moreover $F \mapsto \Phi(F) := \int^\oplus_{\widehat G} F(M_\pi) \,d\GPmeas(\pi)$ is clearly a $*$-homomorphism (by properties of direct integrals and of the functional calculus for the $M_\pi$), which maps the identity function to $M$. Finally, the fact that $\Phi(F_n) \to \Phi(F)$ strongly (as operators on $\int^\oplus_{\widehat G} \HS(\Hilb^\pi) \,d\GPmeas(\pi)$) whenever $F_n \to F$ pointwise boundedly is easily seen by dominated convergence, since $F_n(M_\pi) \to F(M_\pi)$ strongly for all $\pi \in \widehat G$ \cite[\S II.2.3, Proposition 4]{dixmier_vonneumann_1981}.

Recall that, for all $\pi \in \widehat G$, $M_\pi$ is the operator of multiplication on the left by $\pi(\Kern_{E})$. We further claim that $F(M_\pi)$ is the operator of multiplication on the left by $F(\pi(\Kern_{E}))$. Similarly as above, this also follows by uniqueness of the functional calculus, and corresponds to the fact that $F(\pi(\Kern_{E}) \otimes I) = F(\pi(\Kern_E)) \otimes I$ as operators on the Hilbert tensor product $\Hilb^\pi \otimes (\Hilb^\pi)^*$.

By putting all together, we obtain that, for all $f \in L^2(G)$,
\[
\Four (F(E) f) = \left(\int^\oplus_{\widehat G} F(M_\pi) \,d\GPmeas(\pi)\right) \Four f,
\]
that is,
\[
\Four(F(E) f)(\pi) = F(\pi(\Kern_E)) \, \Four(f)(\pi)
\]
for $\GPmeas$-almost all $\pi \in \widehat G$ (note that the set of $\pi$ where this equality holds may depend on $F$ and $f$). Under the assumption $\Kern_{F(E)} \in L^1(G)$ and $f \in L^1 \cap L^2(G)$, we also have $\Four(F(E) f)(\pi) = \pi(f * \Kern_{F(E)}) = \pi(\Kern_{F(E)}) \pi(f)$, and therefore the above identity can be rewritten as
\[
\pi(\Kern_{F(E)}) \pi(f) = F(\pi(\Kern_E)) \pi(f)
\]
for $\GPmeas$-almost all $\pi \in \widehat G$. If we take as $f$ a countable approximate identity on $G$, then $\pi(f)$ converges in the strong operator topology to the identity of $\Hilb^\pi$ for all $\pi\in \widehat G$, and from the last identity we can derive \eqref{eq:rep_fc}.
\end{proof}

One might hope to obtain a proof of statement \eqref{st:rl} by combining Proposition \ref{prp:rep_fc} with Corollary \ref{cor:equivalent_rl}.
 However, in Proposition \ref{prp:rep_fc} we obtain \eqref{eq:rep_fc} for almost all $\pi \in \widehat G$, i.e., up to a $\GPmeas$-null set; instead Corollary \ref{cor:equivalent_rl}\ref{en:change_rep} requires the analogous identity for all $\pi \in \widehat G$, up to changing $F$ on a $\Pmeas_\opL$-null set.

In these respects, it seems relevant to relate the Plancherel measure $\Pmeas_\opL$ associated to $\opL$ and the group Plancherel measure $\GPmeas$. Indeed the following observation allows us to write $\Pmeas_\opL$ as a sort of push-forward of $\GPmeas$ (cf.\ \cite[\S 4.4.1]{martini_multipliers_2010}).

\begin{prp}
For all bounded Borel functions $F : \Rnon \to \CC$ such that $\Kern_{F(\opL)} \in L^1 \cap L^2(G)$,
\begin{equation}\label{eq:bi_plancherel}
\|F\|_{L^2(\Rnon,\Pmeas_\opL)}^2 = \int_{\widehat G} \|F(d\pi(\opL))\|_{\HS}^2 \,d\GPmeas(\pi).
\end{equation}
This applies, in particular, to all $F \in \Sz(\Rnon)$.
\end{prp}
\begin{proof}
By Proposition \ref{prp:rep_fc} and the Plancherel formulas \eqref{eq:opLplancherel} for $\opL$ and \eqref{eq:Gplancherel} for the group Fourier transform,
\begin{multline*}
\|F\|_{L^2(\Rnon,\Pmeas_\opL)}^2 = \|\Kern_{F(\opL)}\|_{L^2(G)}^2 \\
= \int_{\widehat G} \|\pi(\Kern_{F(\opL)})\|_{\HS}^2 \,d\GPmeas(\pi) = \int_{\widehat G} \|F(d\pi(\opL))\|_{\HS}^2 \,d\GPmeas(\pi)
\end{multline*}
and we are done.
\end{proof}

This result will prove useful in obtaining statement \eqref{st:rl} for the group $\Mot$.

\section{Plancherel measure and Riemann--Lebesgue lemma on \texorpdfstring{$\Mot$}{SE(2)}}\label{s:rlmot}

We use the notation of Section \ref{s:se2}.

\begin{prp}
Let $G = \Mot$. For all $K \in L^1 \cap L^2(G)$,
\begin{equation}\label{eq:motplancherel}
\|K\|_{L^2(\Mot)}^2 = \int_0^\infty \|\pi_r(K)\|_{\HS}^2 \,r\,dr.
\end{equation}
In other words, the group Plancherel measure of $\Mot$ is concentrated on the collection of representations $\{\pi_r\}_{r \in \Rpos} \subseteq \widehat G$ and is given by $r \,dr$ in terms of the parameter $r \in \Rpos$.
\end{prp}
\begin{proof}
From \eqref{eq:motrepn} one obtains that, for all $K \in L^1 \cap L^2(G)$,
\[
\|\pi_r(K)\|_{\HS}^2 = \int_\TT \int_\TT |K(\widehat{r e^{i\theta}},e^{i\phi})|^2 \,d\theta \,d\phi.
\]
Integration of the above expression with respect to $r \,dr$, a change of variable and the Plancherel formula for the Euclidean Fourier transform then yield \eqref{eq:motplancherel}.
\end{proof}

\begin{prp}\label{prp:mot_plancherel_ac}
Let $\opL$ be any sub-Laplacian on $\Mot$. Then the Plancherel measure $\Pmeas_\opL$ associated to $\opL$ is absolutely continuous with respect to the Lebesgue measure on $\Rnon$.
\end{prp}
\begin{proof}
As in the proof of Proposition \ref{prp:se2sz}, we distinguish between two cases, according to the form of $\opL$.

In the case $\opL = \Delta$, the Plancherel measure is a given by \eqref{eq:RTPmeas2} (for $n=2$ and $m=1$) and is absolutely continuous with respect to the Lebesgue measure.

Assume instead that $\opL = \opL_0 + \beta \Delta_0$ for some $\beta \in \Rnon$.
From Section \ref{s:se2} we know that $d\pi_r(\opL) = \opML_{r^2} + \beta r^2$.
Let now $\lambda_{(i,j),k}^q$ for $(i,j) \in \{0,1\}^2$ and $k \in \NN$ be the eigenvalues of $\opML_q$ as in Section \ref{s:mathieu}. Then
\[
\|F(d\pi_r(\opL))\|_{\HS}^2 = \sum_{(i,j) \in \{0,1\}^2} \sum_{k \in \NN} |F(\lambda_{(i,j),k}^{r^2}+\beta r^2)|^2 
\]
and from \eqref{eq:bi_plancherel} we obtain the following expression for the Plancherel measure $\Pmeas_\opL$ associated with the sub-Laplacian $\opL$:
\[\begin{split}
\int_{\Rnon} |F(\lambda)|^2 \,d\Pmeas_\opL(\lambda)
 &= \int_0^\infty \sum_{(i,j) \in \{0,1\}^2} \sum_{k \in \NN} |F(\lambda_{(i,j),k}^{r^2}+\beta r^2)|^2 \,r\,dr \\
 &= \frac{1}{2} \int_0^\infty \sum_{(i,j) \in \{0,1\}^2} \sum_{k \in \NN} |F(\lambda_{(i,j),k}^{q}+\beta q)|^2 \,dq 
\end{split}\]
Let now $\psi_{(i,j),k}(q) = \lambda_{(i,j),k}^{q}+\beta q$. Then, by Propositions \ref{prp:eigen_continuous}, \ref{prp:eigen_smooth} and \ref{prp:eigen_asymp}, $\psi_{(i,j),k} : \RR \to \RR$ is a smooth increasing bijection, whose derivative never vanishes. Hence we can rewrite the above formula as follows:
\[\begin{split}
\int_{\Rnon} |F(\lambda)|^2 \,d\Pmeas_\opL(\lambda)
 &= \frac{1}{2} \sum_{(i,j) \in \{0,1\}^2} \sum_{k \in \NN}  \int_{\psi_{(i,j),k}(0)}^\infty |F(\lambda)|^2 \frac{d\lambda}{\psi_{(i,j),k}'(\psi_{(i,j),k}^{-1}(\lambda))} \\
 &= \frac{1}{2} \int_{\psi_{(i,j),k}(0)}^\infty |F(\lambda)|^2 \sum_{\substack{(i,j) \in \{0,1\}^2, \, k\in\NN \\ \lambda > \psi_{(i,j),k}(0) }} \frac{1}{\psi_{(i,j),k}'(\psi_{(i,j),k}^{-1}(\lambda))} \,d\lambda.
\end{split}\]
This shows that the Plancherel measure $\Pmeas_\opL$ associated with $\opL$ is absolutely continuous with respect to the Lebesgue measure on $\Rnon$.
\end{proof}

The following result concludes the proof of Theorem \ref{thm:riemannlebesgue}\ref{en:riemannlebesguemot}.

\begin{prp}
Statement \eqref{st:rl} holds for any sub-Laplacian $\opL$ on $G = \Mot$.
\end{prp}
\begin{proof}
Let $F : \Rnon \to \CC$ be a bounded Borel function such that $\Kern_{F(\opL)} \in L^1(\Mot)$. From Proposition \ref{prp:rep_fc} and the characterisation \eqref{eq:motplancherel} of the group Plancherel measure on $\Mot$ we deduce that, for (Lebesgue) almost all $r \in \Rpos$,
\[
\pi_r(\Kern_{F(\opL)}) = F(d\pi_r(\opL)). 
\]

As in the proof of Proposition \ref{prp:se2sz}, we distinguish between two cases, according to the form of $\opL$. In the case $\opL = \Delta$ the result follows from Proposition \ref{prp:abelianRL}. Assume instead that $\opL = \opL_0 + \beta \Delta_0$ for some $\beta \in \Rnon$, so $d\pi_r(\opL) = \opML_{r^2} + \beta r^2$. If we write $\lambda^q$ in place of $\lambda^q_{(0,0),0}$ and denote by $H^q$ the corresponding normalized eigenfunction of $\opML_q$ (see Section \ref{s:mathieu}), then
\[
\langle \pi_r(\Kern_{F(\opL)}) H^{r^2}, H^{r^2} \rangle = F(\lambda^{r^2}+\beta r^2)
\]
for almost all $r \in \Rpos$; in other words,
\[
F(\lambda^{r^2}+\beta r^2) = G_F(r)
\]
for almost all $r \in \Rpos$, where $G_F : \RR \to \CC$ is defined as in \eqref{eq:kernel}. Since $\Kern_F \in L^1(\Mot)$, from the Euclidean Riemann--Lebesgue lemma and Proposition \ref{prp:eigen_continuous} it follows easily that $G_F$ is continuous. Hence the function $\tilde G_F : \Rnon \to \CC$ defined by $\tilde G_F(q) = G_F(q^{1/2})$ is also continuous and
\[
F(\psi(q)) = \tilde G_F(q)
\]
for almost all $q \in \Rnon$, where $\psi = \psi_{(0,0),0}$.
Since $\psi : \Rnon \to \Rnon$ is a smooth increasing bijection with nowhere vanishing derivative, we deduce that
\begin{equation}\label{eq:multiplier_iden}
F(\lambda) = \tilde G_F(\psi^{-1}(\lambda))
\end{equation}
for (Lebesgue) almost all $\lambda \in \Rnon$. On the other hand, since $\Pmeas_\opL$ is absolutely continuous with respect to the Lebesgue measure by Proposition \ref{prp:mot_plancherel_ac}, we can conclude that \eqref{eq:multiplier_iden} holds for $\Pmeas_\opL$-almost every $\lambda \in \Rnon$, and $\tilde G_F \circ \psi^{-1} : \Rnon \to \CC$ is continuous. By \eqref{eq:vanishinfty} it is then clear that $\tilde G_F \circ \psi^{-1} \in C_0(\Rnon)$, and we are done.
\end{proof}

\section{Appendix: Analysis of a modified Mathieu equation}\label{s:mathieu}

Let $q \in \RR$. Define the differential operator $\opML_q$ on $L^2(\TT)$ by
\[
\opML_q = -\partial_\phi^2 + q \sin^2 \phi.
\]
$\opML_q$ is a self-adjoint Schr\"odinger operator with periodic potential. The spectral theory of $\opML_q$ is discussed in several places in the literature and is related to the theory of Mathieu functions (see, e.g., \cite{wolf_mathieu_2010} and references therein). Here we are particularly interested in the behaviour of eigenvalues and eigenfunctions of $\opML_q$ as functions of the parameter $q$. Since we could not find sufficiently precise information in the literature on the behaviour of $q$-derivatives of eigenvalues and eigenfunctions, here we present an essentially self-contained derivation of the properties that we need, some of which can be also found elsewhere.

Note that $\opML_q$ is invariant under the transformations $\phi \mapsto -\phi$ and $\phi \mapsto \pi + \phi$. We can then decompose $L^2(\TT)$ into four $\opML_q$-invariant subspaces:
\[
L^2(\TT) = \bigoplus_{(i,j) \in \{0,1\}^2} L^2_{(i,j)}(\TT),
\]
where $f \in L^2_{(i,j)}(\TT)$ if and only if
\[
f(-\phi) = (-1)^i f(\phi), \qquad f(\pi+\phi) = (-1)^j f(\phi)
\]
for all $\phi \in \TT$ (in other words, $i$ determines whether $f$ is even or odd, while $j$ determines whether $f$ is $\pi$-periodic or $\pi$-antiperiodic). In this way, we can consider separately the spectral theory of $\opML_q$ restricted to each of the $L^2_{(i,j)}(\TT)$.

\begin{prp}\label{prp:discrete_spectrum}
For all $(i,j) \in \{0,1\}^2$, the restriction of $\opML_q$ to $L^2_{(i,j)}(\TT)$ has a discrete spectrum, made of a strictly increasing sequence of simple eigenvalues
\[
\lambda^q_{(i,j),0} < \lambda^q_{(i,j),1} < \lambda^q_{(i,j),2} < \dots,
\]
and, for all $k \in \NN$, the eigenfunctions of $\opML_q$ in $L^2_{(i,j)}(\TT)$ of eigenvalue $\lambda^q_{(i,j),k}$ have $2(2k + i + |i-j|)$ zeros. 
\end{prp}
\begin{proof}
From the theory of Sturm--Liouville operators \cite[Chapter 8, Theorem 3.1]{coddington_theory_1955}, we deduce that $\opML_q$ on the whole $L^2(\TT)$ has eigenvalues
\[
\mu_0^q < \mu_1^q \leq \mu_2^q < \mu_3^q \leq \mu_4^q < \dots
\]
and corresponding real-valued $L^2(\TT)$-normalised eigenfunctions 
\[ 
K_0^q,K_1^q,K_2^q,K_3^q,K_4^q,\dots,
\]
where $K_j^q$ has 2$\lceil j/2 \rceil$ zeros for all $j \in \NN$.

Since $\opML_q$ is invariant under the shift $\phi \mapsto \phi + \pi$, we can decompose $L^2(\TT)$ into $\pi$-periodic and $\pi$-antiperiodic functions and reconstruct the spectral theory of $\opML_q$ on $L^2(\TT)$ from that of the restrictions of $\opML_q$ to the two subspaces. By identifying $\pi$-periodic and $\pi$-antiperiodic functions on $\TT$ with functions on $[0,\pi]$ satisying suitable boundary conditions (see \cite[Chapter 8, eqs.\ (3.1) and (3.2)]{coddington_theory_1955}) and applying again \cite[Chapter 8, Theorem 3.1]{coddington_theory_1955}, we conclude that $K_j^q$ is $\pi$-periodic or $\pi$-antiperiodic according to whether $\lceil j/2 \rceil$ is even or odd.

Invariance of $\opML_q$ under the inversion $\phi \mapsto -\phi$ leads instead to the decomposition of $L^2(\TT)$ into even and odd functions. It is easily seen that even and odd functions on $\TT$ correspond to functions on $[0,\pi]$ satisfying Neumann and Dirichlet boundary conditions respectively. An application of \cite[Theorem 2.1]{coddington_theory_1955} and a comparison of the number of zeros show that $K_0^q$ is even and that, for all $j \in \NN$, one of $K_{2j+1}^q$ and $K_{2j+2}^q$ is even and the other is odd.

The result follows by combining the above information.
\end{proof}

For all $(i,j) \in \{0,1\}^2$ and $k \in \NN$, let $H_{(i,j),k}^q \in L^2_{(i,j)}(\TT)$ be the real-valued eigenfunction of $\opML_q$ of unit $L^2$ norm, such that either $H_{(i,j),k}^q$ or $\partial_\phi H_{(i,j),k}^q$ is positive at the origin (by uniqueness of solutions to the Cauchy problem for second-order ODEs, no eigenfunction of $\opML_q$ can vanish together with its derivative at any point).

We are now interested in the regularity of eigenvalues and eigenfunctions as functions of the parameter $q$.

\begin{prp}\label{prp:eigen_continuous}
For all $(i,j) \in \{0,1\}^2$ and $k \in \NN$, the map $\RR \ni q \mapsto \lambda_{(i,j),k}^q \in \RR$ is $1$-Lipschitz and increasing, while the map $\RR \ni q \mapsto H_{(i,j),k}^q \in C(\TT)$ is continuous.
Moreover, for all $q \in \RR$,
\begin{equation}\label{eq:eigen_negative_par}
\lambda_{(i,j),k}^{-q} = \lambda^q_{(|i-j|,j),k} -q,
\end{equation}
and, when $q =0$,
\[
\lambda_{(i,j),k}^0 = (2k + i + |i-j|)^2.
\]
\end{prp}
\begin{proof}
Note that, by the minimax principle,
\begin{equation}\label{eq:minimax}
\begin{split}
\lambda_{(i,j),k}^q &= \min_{\substack{f \in L^2_{(i,j)}(\TT) \\ \|f\|_2 = 1, \, f\perp H_{(i,j),0}^q,\dots,H_{(i,j),k-1}^q}} \langle \opML_q f,f \rangle \\
&= \max_{g_1,\dots,g_k \in L^2_{(i,j)(\TT)}} \min_{\substack{f \in L^2_{(i,j)}(\TT) \\ \|f\|_2 = 1, \, f\perp g_1,\dots,g_k}} \langle \opML_q f,f \rangle.
\end{split}
\end{equation}
The second expression for $\lambda_{(i,j),k}^q$, together with the fact that $\opML_q - \opML_{q'} \geq 0$ whenever $q \geq q'$, immediately implies that $\lambda_{(i,j),k}^q$ is an increasing function of $q$. Similarly, since $\opML_q - \opML_{q'} = (q-q') \sin^2 \phi$ is a bounded operator on $L^2(\TT)$ for all $q,q' \in \RR$, with norm bounded by $|q-q'|$, from \eqref{eq:minimax} we deduce that
\begin{equation}\label{eq:eigen_lipschitz}
|\lambda_k^q - \lambda_k^{q'}| \leq |q-q'|,
\end{equation}
i.e., $q \mapsto \lambda_k^q$ is $1$-Lipschitz.

For all $q,\lambda \in \RR$, let $\Phi(q,\lambda,\cdot)$ denote the solution to the Cauchy problem
\[
\begin{cases}
(-\partial_t^2 +q\sin^2 t - \lambda) \Phi(q,\lambda,t) = 0 \\
\Phi(q,\lambda,0) = 1-i\\
\partial_t \Phi(q,\lambda,0) = i
\end{cases}
\]
on $\RR$.
Note that, by the Cauchy--Kowalevski theorem \cite[Theorem (1.25)]{folland_introduction_1995}, $\Phi(q,\lambda,t)$ 
is an analytic function of $(q,\lambda,t) \in \RR^3$. Note moreover that, for all $k \in \NN$,
\begin{equation}\label{eq:cont_dependence_parameters}
H_{(i,j),k}^q(e^{i\phi}) =
\frac{\Phi(q,\lambda_{(i,j),k}^q,\phi)}{\left( \int_0^1 \Phi(q,\lambda_{(i,j),k}^q,2\pi t)^2 \,dt \right)^{1/2}}.
\end{equation}
From the continuity of $q \mapsto \lambda_{(i,j),k}^q$ we then deduce immediately the continuity of $q \mapsto H_{(i,j),k}^q$.

Note now that
\begin{equation}\label{eq:op_negative_par}
\opML_{-q} = -\partial_\phi^2 - q \sin^2 \phi = -\partial_\phi^2 + q \cos^2 \phi - q,
\end{equation}
which shows that the shift $\phi \mapsto \phi+\pi/2$ maps $\opML_{-q}$ into $\opML_{q} -q$. Considerations on the behaviour of parity and $\pi$-periodicity under this shift immediately yield \eqref{eq:eigen_negative_par}.

Finally, the value of $\lambda_{(i,j),k}^0$ is easily determined since $\opML_0 = -\partial_\phi^2$, whose eigenfunctions and eigenvalues in $L^2(\TT)$ are well known.
\end{proof}

Eigenvalues and eigenfunctions are actually smooth functions of the parameter $q$, as we now show.

In this and the following results we will work with a fixed parity/periodicity $(i,j) \in \{0,1\}$, so in their proofs we will drop $(i,j)$ from the notation and just write $\lambda_{k}^q$ and $H_{k}^q$ instead of $\lambda_{(i,j),k}^q$ and $H_{(i,j),k}^q$.

\begin{prp}\label{prp:eigen_smooth}
The maps $\RR \ni q \mapsto \lambda_{(i,j),k}^q \in \RR$ and $\RR \ni q \mapsto H_{(i,j),k}^q \in C(\TT)$ are infinitely differentiable. Moreover
\begin{equation}\label{eq:derivative_eigen_k}
\partial_q \lambda_{(i,j),k}^q = \int_{\TT} \sin^2 \phi \, (H_{(i,j),k}^q(e^{i\phi}))^2 \,d\phi
\end{equation}
and the latter expression is decreasing in $q$ when $k =0$.
\end{prp}
\begin{proof}
From \eqref{eq:minimax} we deduce that, for all $q,q' \in \RR$,
\[\begin{split}
\lambda_k^{q'} &\leq \frac{\left\langle \opML_{q'} (H_k^q - \sum_{j=0}^{k-1} \langle H_k^q,H_j^{q'} \rangle H_j^{q'} ), H_k^q - \sum_{j=0}^{k-1} \langle H_k^q,H_j^{q'} \rangle H_j^{q'} \right\rangle}{\left\| H_k^q - \sum_{j=0}^{k-1} \langle H_k^q,H_j^{q'} \rangle H_j^{q'}\right\|^2} \\
&= \frac{\left\langle \opML_{q'} H_k^q , H_k^q - \sum_{j=0}^{k-1} \langle H_k^q,H_j^{q'} \rangle H_j^{q'} \right\rangle}{1 - \sum_{j=0}^{k-1} \langle H_k^q,H_j^{q'} \rangle^2} \\
&= \frac{\lambda_k^q + \left\langle (\opML_{q'}-\opML_q) H_k^q , H_k^q \right\rangle - \sum_{j=0}^{k-1} \lambda_j^{q'} \langle H_k^q,H_j^{q'} \rangle^2 }{1 - \sum_{j=0}^{k-1} \langle H_k^q,H_j^{q'} \rangle^2} ,
\end{split}\]
which gives
\[
\lambda_k^{q'}-\lambda_k^q \leq \sum_{j=0}^{k-1} (\lambda_k^{q'} - \lambda_j^{q'}) \langle H_k^q,H_j^{q'}-H_j^q \rangle^2 + (q'-q) \int_\TT \sin^2 \phi \, H_k^q(\phi)^2 \,d\phi
\]
and, for $q < q'$,
\begin{multline}\label{eq:lip_eigen}
\sum_{j=0}^{k-1} (\lambda_k^{q} - \lambda_j^{q}) \frac{\langle H_k^{q'},H_j^{q}-H_j^{q'} \rangle^2}{q-q'} + \int_\TT \sin^2 \phi \, H_k^{q'}(\phi)^2 \,d\phi \leq \frac{\lambda_k^{q'}-\lambda_k^q}{q'-q} \\
\leq \sum_{j=0}^{k-1} (\lambda_k^{q'} - \lambda_j^{q'}) \frac{\langle H_k^q,H_j^{q'}-H_j^q \rangle^2}{q'-q} + \int_\TT \sin^2 \phi \, H_k^q(\phi)^2 \,d\phi.
\end{multline}
Note that, in the particular case where $k=0$, the sums on $j$ disappear and these inequalities imply that $q \mapsto \int_\TT \sin^2 \phi \, H_0^q(\phi)^2 \,d\phi$ is decreasing.

We can now proceed with an inductive argument on $k$ to prove smoothness of $q \mapsto H_k^q$ and $q \mapsto \lambda_k ^q$. Indeed assume that $q \mapsto H_j^q$ and $q \mapsto \lambda_j ^q$ are infinitely differentiable for $j=0,\dots,k-1$. Then
\[
\frac{\langle H_k^q,H_j^{q'}-H_j^q \rangle^2}{q'-q} = \langle H_k^q,H_j^{q'}-H_j^q \rangle \left \langle H_k^q,\frac{H_j^{q'}-H_j^q}{q'-q} \right\rangle,
\]
which tends to $0$ as $q' \searrow q$ and as $q \nearrow q'$.  
If we take the corresponding limits in \eqref{eq:lip_eigen}, we obtain that the map $q \mapsto \lambda_k^q$ is differentiable, with derivative given by \eqref{eq:derivative_eigen_k}.
A repeated application of \eqref{eq:cont_dependence_parameters} and \eqref{eq:derivative_eigen_k} inductively yields that $q \mapsto \lambda_k^q$ and $q \mapsto H_k^q$ are infinitely differentiable.
\end{proof}

We are now interested in the asymptotic behaviour as $q \to \infty$. 
The underlying idea to obtain such asymptotic results is the fact that, through a suitable rescaling, the eigenvalue equation for the Mathieu operator tends to the one for the Hermite operator $-\partial_t^2 + t^2$ on $\RR$.

We start with a preliminary estimate, whose proof follows \cite{ince_mathieu_1927}.

\begin{lem}
For all $(i,j) \in \{0,1\}^2$ and $k \in \NN$,
\begin{equation}\label{eq:asymp_eigen_est}
\limsup_{q \to \infty} \frac{\lambda_{(i,j),k}^q}{q^{1/2}} \leq 2(2k+i)+5.
\end{equation}
\end{lem}
\begin{proof}
Let $q > 0$. Define a function $U_k^q : \RR \to \CC$ by
\begin{equation}\label{eq:rescaled_solution}
U_k^q(t) = \begin{cases}
\pi^{-1/2} q^{-1/8} H_k^q(e^{it/q^{1/4}}) &\text{if $t \in (-q^{1/4}\pi/2,q^{1/4}\pi/2)$,}\\
0 &\text{otherwise.}
\end{cases}
\end{equation}
Then on the interval $I_q = (-q^{1/4}\pi/2,q^{1/4}\pi/2)$ the function $U_k^q$ is smooth and satisfies
\begin{equation}\label{eq:rescaled_diffeq}
-\partial_t^2 U_k^q(t) + q^{1/2} \sin^2(t/q^{1/4}) U_k^q(t) = q^{-1/2} \lambda_k^q U_k^q(t).
\end{equation}
Moreover the number of zeros of $U_k^q$ in $I_q$ is the same as the number of zeros of $\phi \mapsto H_k^q(e^{i\phi})$ in $(-\pi/2,\pi/2)$, that is, $2k+i$ (see Proposition \ref{prp:discrete_spectrum}).

Note that, by \eqref{eq:eigen_lipschitz}, $\lambda_k^q \leq q + \lambda_k^0$ for all $q > 0$, hence
there exists $\bar q_k \in \Rpos$ such that
\begin{equation}\label{eq:interval_comparison}
\lambda_k^q/q^{1/2} < (q^{1/4} \pi/2)^2
\end{equation}
for all $q > \bar q_k$. In order to conclude, it will be enough to show that $\lambda_k^q/q^{1/2} \leq 2(2k+i)+5$ for all $q > \bar q_k$.

Let $q > \bar q_k$. We may also assume that $\lambda_k^q/q^{1/2} > 1$ (otherwise there is nothing to prove). Let $m = \lceil (\lambda_k^q/q^{1/2}-1)/2 \rceil-1$; in other words, $2m+1$ is the greatest odd number strictly less than $\lambda_k^q/q^{1/2}$. Let $h_{m}$ denote the $m$th Hermite function, which satisfies
\begin{equation}\label{eq:hermite_eq}
-\partial_t^2 h_{m}(t) + t^2 h_{m} = (2m + 1) h_{m}.
\end{equation}
Since $q^{1/2} \sin^2(t/q^{1/4}) \leq t^2$ and $\lambda_k^q/q^{1/2} > 2m+1$, we can use the Sturm comparison theorem \cite[Chapter 8, Theorem 1.1]{coddington_theory_1955} to compare zeros of solutions to the differential equations \eqref{eq:rescaled_diffeq} and \eqref{eq:hermite_eq}; in particular, on the interval $I_q$ the function $h_m$ has at most one more zero than $U_k^q$. On the other hand, the Hermite function $h_m$ has exactly $m$ zeros on $\RR$, all of which lie in the interval $(-\sqrt{2m+1},\sqrt{2m+1})$; by \eqref{eq:interval_comparison}, this interval is contained in $I_q$ and therefore
\[
m-1 \leq 2k+i;
\]
hence
\[
\lambda_k^q / q^{1/2} \leq 2m+3 \leq 2(2k+i)+5
\]
and we are done.
\end{proof}

We now obtain precise asymptotic information as $q \to \infty$ for the eigenvalues and their $q$-derivative. The asymptotic \eqref{eq:asymp_eigen} can be found elsewhere in the literature (see, e.g., \cite[\S 2.331]{meixner_mathieusche_1954} and \cite[\S 5.2.1]{arscott_periodic_1964}).

\begin{prp}\label{prp:eigen_asymp}
For all $(i,j) \in \{0,1\}^2$ and $k \in \NN$,
\begin{equation}\label{eq:asymp_eigen}
\lim_{q \to \infty} \frac{\lambda_{(i,j),k}^q}{q^{1/2}} = 2(2k+i)+1
\end{equation}
and
\begin{equation}\label{eq:asymp_der_eigen}
\lim_{q \to \infty} \frac{q \, \partial_q \lambda_{(i,j),k}^q}{\lambda_{(i,j),k}^q} = \frac{1}{2}.
\end{equation}
\end{prp}
\begin{proof}
Let $q \geq 1$. Define the interval $I_q$ and the function $U_k^q : \RR \to \CC$ as in \eqref{eq:rescaled_solution}.
Then
\[
\|U_k^q\|_{L^2(\RR)} = \|H_k^q\|_{L^2(\TT)} = 1
\]
and
\[
\|\partial_t U_k^q\|_{L^2(I_q)} = q^{-1/4} \|\partial_\phi H_k^q\|_{L^2(\TT)} \leq q^{-1/4} \langle \opML_q H_k^q, H_k^q\rangle^{1/2} = (\lambda_k^q / q^{1/2})^{1/2} \lesssim_k 1
\]
by \eqref{eq:asymp_eigen_est}. Since $I_q \supseteq K_r := \overline{I_r}$ if $q > r$,
 this shows that $\{ U_k^q|_{K_r} \tc q > r\}$ is bounded in $W^{1,2}(K_r)$ for all $r \geq 1$.

Note further that, for all $\epsilon > 0$,
\[\begin{split}
\int_{\{|t| \geq \epsilon^{-1/2}\}} |U_k^q(t)|^2 \,dt &\leq \int_{\{|\sin (t/q^{1/4})| \geq \frac{2}{\pi} \epsilon^{-1/2} q^{-1/4}\} \cap I_q} |U_k^q(t)|^2 \,dt \\
&= \pi \int_{\{|\sin (\phi)| \geq \frac{2}{\pi} \epsilon^{-1/2} q^{-1/4}\}} |H_k^q(e^{i\phi})|^2 \,d\phi \\
&\leq \epsilon \frac{\pi^2}{2} q^{1/2} \int_{\TT} \sin^2 \phi \, |H_k^q(e^{i\phi})|^2 \,d\phi \leq \epsilon \frac{\pi^2}{2} \frac{\lambda_k^q}{q^{1/2}}.
\end{split}\]
Since $\lambda_k^q \lesssim_k q^{1/2}$ by \eqref{eq:asymp_eigen_est}, we conclude that, for all $\epsilon > 0$,
\[
\int_{\{|t| \geq \epsilon^{-1/2}\}} |U_k^q(t)|^2 \,dt \lesssim_k \epsilon,
\]
hence $\{ |U_k^q|^2 : q \geq 1\}$ is tight.

Let $(q_\ell)_\ell$ be any sequence in $[1,\infty)$ with $q_\ell \to \infty$. Since $\lambda_k^q/q^{1/2} \lesssim_k 1$ and the embedding $W^{1,2}(K_r) \subseteq C(K_r)$ is compact for all $r \geq 1$, up to extraction of a subsequence we may assume that $\lambda_k^{q_\ell}/q_\ell^{1/2} \to \lambda_k \in \Rnon$ and $U_k^{q_\ell} \to U_k \in C(\RR)$ uniformly on compacta. Tightness then yields that $U_k^{q_\ell} \to U_k$ in $L^2(\RR)$ as well, and in particular $\|U_k\|_{L^2(\RR)} = 1$. Moreover, since $U_k^q$ satisfies the differential equation \eqref{eq:rescaled_diffeq}, we deduce that $\partial_t^2 U_k^{q_\ell}$ also converges uniformly on compacta, which in turn implies that $\partial_t U_k^{q_\ell}$ converges uniformly on compacta as well; repeated differentiation of the differential equation \eqref{eq:rescaled_diffeq} actually gives that any derivative of $U_k^q$ converges uniformly on compacta. In conclusion $U_k \in C^\infty(\RR)$ and satisfies the limit equation
\[
-\partial_t^2 U_k(t) + t^2 U_k(t) = \lambda_k U_k(t).
\]

Since $\|U_k\|_2 = 1$, $U_k$ and $\lambda_k$ are an eigenfunction and an eigenvalue of the Hermite operator $-\partial_t^2 + t^2$. This implies that $\lambda_k = 2m+1$ for some $m \in \NN$, and moreover $U_k$ has exactly $m$ zeros, which are all simple. Since 
$U_k^{q_\ell} \to U_k$ 
uniformly on compacta, it is easily seen that, for $\ell$ sufficiently large, $U_k^{q_\ell}$ in $I_{q_\ell}$ has at least as many zeros as $U_k$ in $\RR$, i.e., $m \leq 2k+i$ and $\lambda_k \leq 2(2k+i) +1$.

Note now that, up to further extraction of subsequence, we may assume that the same convergence results also hold with $k$ replaced with every $k' \leq k$: in particular, $\lambda_{k'}^{q_\ell} \to \lambda_{k'}$, $U_{k'}^{q_\ell} \to U_{k'}$ in $L^2(\RR)$ and $U_{k'}$ is eigenfunction of the Hermite operator of eigenvalue $\lambda_{k'} \leq 2(2k'+i)+1$. Clearly $\lambda_{k'}^q < \lambda_{k''}^q$ when $k' < k''$, so at the limit $\lambda_{k'} \leq \lambda_{k''}$; on the other hand
\[
\langle U_{k'}^q, U_{k''}^q \rangle_{L^2(\RR)} = \langle H_{k'}^q, H_{k''}^q \rangle_{L^2(\TT)} = 0,
\]
so $\langle U_{k'}, U_{k''} \rangle_{L^2(\RR)} = 0$ and $\lambda_{k'} < \lambda_{k''}$. Note further that, since the $U_{k'}$ have the same parity as the $H_{k'}^q$, the $\lambda_{k'}$ belong to the set $\{2(2m+i) +1 \tc m \in \NN\}$ of eigenvalues of the Hermite operator corresponding to eigenfunctions with appropriate parity. Hence the inequalities $\lambda_{k'} \leq 2(2k'+i)+1$ can only be satisfied when equality holds for all $k' \leq k$, and in particular $\lambda_k = 2(2k+i)+1$.

Since the limit $\lambda_k$ of $\lambda_k^{q_\ell}/q^{1/2}$ does not depend on the subsequence $q_\ell$, we conclude that \eqref{eq:asymp_eigen} holds. Similarly, since there is only one $L^2$-normalised eigenfunction $U_k$ of the Hermite operator with eigenvalue $\lambda_k$ and such that either the function or its derivative is positive at the origin, we conclude that $U_k^q \to U_k$ uniformly on compacta, together with its derivatives, and in $L^2(\RR)$.

Finally
\[
q^{1/2} \int_\TT \sin^2 \phi \, (H_k^q(\phi))^2 \,d\phi = q^{1/2} \int_\RR \sin^2(t/q^{1/4}) \, (U_k^q(t))^2 \, dt ,
\]
hence, by Fatou's lemma,
\[
\liminf_{q \to \infty } q^{1/2} \int_\TT \sin^2 \phi \, (H_k^{q}(\phi))^2 \,d\phi \geq  \int_\RR t^2 \, (U_k(t))^2 \, dt = \frac{\lambda_k}{2} \int_\RR (U_k(t))^2 \, dt = \frac{\lambda_k}{2},
\]
where the virial theorem for the harmonic oscillator was used. 
Similarly, since
\[
q^{-1/2} \|\partial_\phi H_k^q\|_{L^2(\TT)}^2 = \|\partial_t U_k^q\|_{L^2(I_q)}^2
\]
and $\partial_t U_k^{q} \to \partial_t U_k$ uniformly on compacta, we conclude that
\[
\liminf_{q \to \infty} q^{-1/2} \|\partial_\phi H_k^{q}\|_{L^2(\TT)}^2 \geq \|\partial_t U_k\|_{L^2(\RR)}^2 = \frac{\lambda_k}{2} \|U_k\|_{L^2(\RR)}^2 = \frac{\lambda_k}{2}.
\]
On the other hand
\[
q^{-1/2} \|\partial_\phi H_k^q\|_{L^2(\TT)}^2 + q^{1/2} \int_\TT \sin^2 \phi \, (H_k^q(\phi))^2 \,d\phi = q^{-1/2} \langle \opML_q H_k^q, H_k^q \rangle = \frac{\lambda_k^q}{q^{1/2}}
\]
and $\lambda_k^{q}/q^{1/2} \to \lambda_k$, hence
\[
\limsup_{q \to \infty} \, q^{1/2} \int_\TT \sin^2 \phi \, (H_k^q(\phi))^2 \,d\phi \leq \lambda_k - \liminf_{q \to \infty} q^{-1/2} \|\partial_\phi H_k^q\|_{L^2(\TT)}^2 \leq \frac{\lambda_k}{2}.
\]
This shows that
\[
q^{1/2} \int_\TT \sin^2 \phi \, (H_k^{q}(\phi))^2 \,d\phi \to \frac{\lambda_k}{2}
\]
and, since $\lambda_k^{q}/q^{1/2} \to \lambda_k$, by \eqref{eq:derivative_eigen_k} we conclude that
\[
\frac{q \,\partial_q \lambda_k^{q}}{\lambda_k^{q}}= \frac{q}{\lambda_k^{q}} \int_\TT \sin^2 \phi \, (H_k^{q}(\phi))^2 \,d\phi \to \frac{1}{2},
\]
which is \eqref{eq:asymp_der_eigen}.
\end{proof}

The previously obtained information finally allows us to obtain the following estimates for higher-order derivatives.

\begin{prp}\label{prp:eigen_estimates_poly}
For all $(i,j) \in \{0,1\}^2$ and $k,N \in \NN$, there exist $a(N),b(N) \in \NN$ such that
\[
|\partial_q^N \lambda_{(i,j),k}^q| \lesssim_{N,k} (1+|q|)^{a(N)}, \qquad \|\partial_q^N H_{(i,j),k}^q\|_2 \lesssim_{N,k} (1+|q|)^{b(N)}
\]
for all $q \in \RR$.
\end{prp}
\begin{proof}
In view of \eqref{eq:eigen_negative_par} and \eqref{eq:op_negative_par}, it is enough to consider the case where $q \geq 0$.

For all $N \in \NN \setminus \{0\}$, $N$-times differentiation of the eigenvalue equation gives
\begin{equation}\label{eq:diff_eigenvalue_eq}
\sum_{j=0}^{N-1} \binom{N}{j} [\partial_q^{N-j} (\opML_q-\lambda_k^q)] \partial_q^j H_k^q + (\opML_q - \lambda_k^q) \partial_q^N H_k^q = 0;
\end{equation}
here $\partial_q^{m} (\opML_q-\lambda_k^q)$ denotes the multiplication operator by $\partial_q^m (q \sin^2 \phi - \lambda_k^q)$ for all $m > 0$. This tells us that the first summand $\sum_{j=0}^{N-1} \binom{N}{j} [\partial_q^{N-j} (\opML_q-\lambda_k^q)] \partial_q^j H_k^q$ is in the range of $\opML_q - \lambda_k^q$, so it is orthogonal to the kernel $\CC H_k^q$ of $\opML_q -\lambda_k^q$. Moreover $\opML_q - \lambda_0^k$ restricted to $(H_0^q)^\perp$ is invertible; hence \eqref{eq:diff_eigenvalue_eq} allows us to determine the component of $\partial_q^N H_k$ orthogonal to $H_k^q$. On the other hand, $N$-times differentiation of the $L^2$-normalization equation
\begin{equation}\label{eq:normalization}
\|H_k^q\|_2^2 = \langle H_k^q, H_k^q \rangle = 1
\end{equation}
gives that
\[
2 \langle \partial_q^N H_k^q, H_k^q \rangle + \sum_{j=1}^{N-1} \binom{N}{j} \langle \partial_q^j H_k^q, \partial_q^{N-j} H_k^q \rangle = 0,
\]
which allows us to determine the component of $\partial_q^N H_k^q$ along $H_k^q$. In conclusion
\[\begin{split}
\partial_q^N H_k^q = &- \frac{1}{2} \sum_{j=1}^{N-1} \binom{N}{j} \langle \partial_q^j H_k^q, \partial_q^{N-j} H_k^q \rangle H_k^q \\
&- (\opML_q - \lambda_k^q)|_{(H_k^q)^\perp}^{-1} \sum_{j=0}^{N-1} \binom{N}{j} [\partial_q^{N-j} (\opML_q-\lambda_k^q)] \partial_q^j H_k^q.
\end{split}\]

Note now that $\partial_q (\opML_q -\lambda_k^q) = \sin^2 \phi - \partial_q \lambda_k^q$ and $\partial_q^m  (\opML_q -\lambda_k^q) = -\partial_q^m \lambda_k^q$ for $m > 1$.
The last equation gives, for all $N \in \NN \setminus \{0\}$, the estimate
\begin{equation}\label{eq:est_der_eigenf}
\begin{split}
\|\partial_q^N H_k^q\|_2 
&\lesssim_{N,k} \sum_{j=1}^{N-1} \|\partial_q^j H_k^q\|_2 \|\partial_q^{N-j} H_k^q\|_2 \\
&\qquad + \sum_{j=0}^{N-1} (1+ |\partial_q^{N-j} \lambda_k^q|) \, \|\partial_q^j H_k^q\|_2,
\end{split}
\end{equation}
where we have also used the fact that $\inf_{m \tc m \neq k}|\lambda_m^q-\lambda_k^q| \gtrsim_k (1+q)^{1/2} \geq 1$ by \eqref{eq:asymp_eigen}.
Moreover differentiation of the formula \eqref{eq:derivative_eigen_k} for $\partial_q \lambda_k^q$ gives, for all $N \in \NN$,
\[
\partial_q^{N+1} \lambda_k^q = \sum_{j=0}^N \binom{N}{j} \int_\TT \sin^2 \phi \, \partial_q^j H_k^q(\phi) \, \partial_q^{N-j} H_k^q(\phi) \,d\phi,
\]
which yields
\begin{equation}\label{eq:est_der_eigenv}
|\partial_q^{N+1} \lambda_k^q| \lesssim_N \sum_{j=0}^N \|\partial_q^j H_k^q\|_2 \|\partial_q^{N-j} H_k^q\|_2.
\end{equation}
Finally, note that,
by \eqref{eq:asymp_eigen_est},
\begin{equation}\label{eq:est_first_eigenv}
\lambda_k^q \lesssim_k (1+q)^{1/2}.
\end{equation}

The conclusion follows inductively, by repeated application of the estimates  \eqref{eq:normalization}, \eqref{eq:est_der_eigenf}, \eqref{eq:est_der_eigenv},and \eqref{eq:est_first_eigenv}.
\end{proof}

\end{document}